\documentclass[11pt,oneside]{amsart}
\usepackage{graphicx}
\usepackage{color}
\usepackage{epsfig}
\usepackage[english]{babel}
\usepackage[utf8]{inputenc}
\usepackage{amsmath}
\usepackage{amsthm}
\usepackage{amsfonts}
\usepackage{graphicx}
\usepackage{bbm}
\usepackage[colorinlistoftodos]{todonotes}
\usepackage{subfigure}
\usepackage{epstopdf}
\def\dd{\mathrm{d}}
\def\R{\mathbb{R}}

\newtheorem{thm}{Theorem}

\newtheorem{lem}[thm]{Lemma}
\newtheorem{rmk}[thm]{Remark}
\newtheorem{deff}[thm]{Definition}

\newtheorem{prop}[thm]{Proposition}

\numberwithin{equation}{section}
\numberwithin{thm}{section}

\DeclareMathOperator*{\esssup}{ess\,sup}

\begin{document}

\title{The Derivation of the Linear Boltzmann Equation from a Rayleigh Gas Particle Model}

\author{Karsten Matthies}
\address{Department of Mathematical Sciences, University of Bath, Bath BA2
7AY, United Kingdom} \email{K.Matthies@bath.ac.uk}

\author{George Stone}
\address{Department of Mathematical Sciences, University of Bath, Bath BA2
7AY, United Kingdom} \email{G.Stone@bath.ac.uk}

\author{Florian Theil}
\address{Mathematics Institute, University of Warwick, Coventry CV4
7AL, United Kingdom} \email{f.theil@warwick.ac.uk}

\begin{abstract}
A linear Boltzmann equation is derived in the Boltzmann-Grad scaling
for the deterministic dynamics of many
interacting particles with random initial data. We study a Rayleigh gas where a tagged particle is undergoing hard-sphere collisions with background particles, which do not interact among each other. In the Boltzmann-Grad scaling, we derive the validity of a linear Boltzmann equation for arbitrary long times under moderate assumptions on higher moments of the initial distributions of the tagged particle and the possibly non-equilibrium distribution of the background. The convergence of the empiric dynamics to the Boltzmann dynamics is shown using Kolmogorov equations for associated probability measures on collision histories.
\end{abstract}

\maketitle

\section{Introduction}

The derivation of continuum equations from atomistic particle models is currently a major problem in mathematical physics with origins in Hilbert's Sixth Problem. A particular interest in this area is the derivation of the Boltzmann equation from atomistic particle dynamics. The first major work in this area was by Lanford \cite{lanford75} which showed convergence from a hard-sphere particle model for short times by using the Bogoliubov-Born-Green-Kirkwood-Yvon (BBGKY) hierarchy, see e.g. \cite{born46,cercignani94,uchiyama1988}. A recent major work by Gallagher, Saint-Raymond and Texier \cite{saintraymond13} continued the BBGKY development and proved convergence to the Boltzmann equation for short times for both hard-sphere and short-range potentials. The latter of which was improved by Pulvirenti \cite{pul13}. A particular difficulty has been proving convergence for arbitrarily long times, since the aforementioned results hold only for times up to a fraction of the mean free flight time. In \cite{bod15brown} Bodineau, Gallagher and Saint-Raymond were able to utilise the tools from \cite{saintraymond13} to prove convergence from a hard-sphere particle model to the linear Boltzmann equation for arbitrary long times in the case that the initial distribution of the background is near equilibrium. They were further able to use the linear Boltzmann equation as an intermediary step to prove convergence to Brownian motion. Further in a following paper \cite{bod15} the authors were able to consider weaker assumptions on the initial data and prove convergence to the Stokes-Fourier equations. For a general overview of the Boltzmann equation and the BBGKY hierarchy we refer to the books \cite{cerci88,cercignani94,spohn91}.

A new method to tackle this problem has recently been developed in a series of papers \cite{matt08,matthies10,matt12}. This method employs semigroup techniques to study the evolution of collision trees rather than the BBGKY hierarchy. This comes from studying the distribution of the history of the particles up to a certain time rather than the distribution of the particles at a specific time. These papers have been able to prove convergence for arbitrary times but only for a simplified particle interaction system. This is the approach that we continue to develop in this paper.

\subsection{The BBGKY Hierarchy}

The standard approach to tackling the derivation of the Boltzmann equation is via the BBGKY hierarchy. We refer specifically to \cite{bod15brown} for hard spheres. If $f_N(t)$ represents the $N$ particle distribution resulting from hard sphere dynamics at time $t$ then, away from collisions, $f_N$ satisfies the Liouville equation,
\[ \partial_t f_N(t) + v \cdot \partial_x f_N(t)= 0. \]
Considering this equation in weak form and integrating we have, away from collisions,
\begin{equation} \label{eq-bbgky}
 \partial_t f_N^{(s)}(t) + \sum_{1 \leq i \leq s } v_i \cdot \nabla _{x_i} f_N(t) = C_{s,s+1}f_N ^{(s+1)}(t),
\end{equation}
for $s=1,\dots,N$ where $f_N^{(s)}$ denotes the $s$ particle marginal and where $C_{s,s+1}$ is the effect on the distribution of the first $s$ particles by a collision with another particle given explicitly in \eqref{eq-Chs} below. The system of $N$ equations \eqref{eq-bbgky} is known as the BBGKY hierarchy. If the initial distribution of the $N$ particles is given by,
$$f_N(0)=\frac{1}{\mathcal{Z}_N} \mathbbm{1}_{\mbox{no}} f_0^{\otimes N},$$
where $\mathbbm{1}_{\mbox{no}}$ conditions on no initial overlap and $\mathcal{Z}_N$ is a normalising constant, then the initial distribution of $f_N^{(s)}$ is given by,
\[ f_N^{s}(0,Z_s) = \int f_N(0,Z_n) \, \mathrm{d}z_{s+1} \dots \, \mathrm{d}z_{N}.  \]
After successive time integration of \eqref{eq-bbgky} one obtains a representation in the form of a finite sum
\begin{align} \label{Duhamel}
f^{(s)}_N(t) = &\sum_{k=0}^{N-s} \int_0^t\int_0^{t_1}\ldots
\int_0^{t_{k-1}} {\bf T}_s(t-t_1){\mathcal C}_{s,s+1}{\bf T}_{s+1}
(t_1-t_2){\mathcal C}_{s+1,s+2}\ldots\\
\nonumber
&\hspace{5cm}\ldots {\bf T}_{s+k}(t_k)f_N^{(s+k)}(0)
\, \dd t_k\ldots \dd t_1,
\end{align}
where ${\bf T}_s$ is the flow map of $s$ hard spheres.

A popular method to establish the convergence of solutions of \eqref{eq-bbgky} is to first demonstrate that
the mild form of the BBGKY hierarchy is given by a contracting operator if $t$ is sufficiently small. This step ensures that the
sum \eqref{Duhamel} converges absolutely, and in a second step one checks the convergence of the individual
terms.
It is noteworthy that the necessity to establish the contraction property of the mild BBGKY operator (and thereby
restricting the analysis to small values of $t$) is
due to the fact that the individual terms in \eqref{Duhamel} are unsigned.

The key achievement of this paper is to demonstrate that in a simpler setting a representation formula similar to
\eqref{Duhamel} for $f_N^{(s)}$ can be found so that the
individual terms are non-negative and can be interpreted
as probabilities. This representation offers two significant
advantages:
\begin{enumerate}
\item It is possible to allow irregular initial data which
may not be exponentially tight.\\
\item The need to establish a contraction property is replaced by a tightness bound, which is related to the
properties of physically relevant objects such as collision
histories. In particular, convergence can be
established for all times.
\end{enumerate}

\subsection{The Lorentz and Rayleigh Gas}

Instead of considering a system of $N$ identical hard spheres evolving via elastic collisions one can consider a single tagged or tracer particle evolving among a system of fluid scatterers or background particles.

If the background particles are fixed and of infinite relative mass to the tracer particle then one has a model known as the Lorentz gas first introduced by Lorentz in \cite{lor05} to study the motion of electrons in a metal.

Much research has been done deriving the linear Boltzmann equation from a Lorentz gas with randomly placed scatterers, for example \cite{bold83,gall99,spohn78} and a large number of references found in Part I Chapter 8 of \cite{spohn91}. The linear Boltzmann equation can however fail to hold if we consider non-random periodic scatterers, as shown for example by Golse \cite{golse08} and Marklof \cite{marklof10}. The existence of a limiting stochastic process for the periodic Lorentz gas from the Boltzmann-Grad limit was shown by Marklof and Strömbergsson in \cite{mark11}.

When a force field is present the convergence of the distribution of the tracer particle in an absorbing Lorentz gas to the solution of a gainless linear Boltzmann equation was proved in \cite{desvi07}. The authors also proved that if the scatterers move with a constant random initial velocity then the convergence can be proven with significantly weaker assumptions on the force field.

Closely related to the Lorentz gas is the Rayleigh gas, where the background particles are no longer of infinite mass. Convergence to Brownian motion is discussed in Part I Chapter 8 of \cite{spohn91}. In \cite{lebo82} Lebowitz and Spohn proved the convergence of the momentum process for a test particle to a jump process associated to the linear Boltzmann equation. This was proved for arbitrarily long times, via the BBGKY hierarchy, when the initial distribution of the velocities is at equilibrium. This builds on their previous work  \cite{lebo78,lebo82b,bei80}.

In this paper we consider a Rayleigh gas where the background particles are of equal mass to the tagged particle and have no self interaction. The particles evolve via a simplified form of hard-sphere dynamics whereby the background particles do not change velocity. We consider non-equilibrium initial data but require that the background particles are spatially homogeneous. Convergence is proved for arbitrary times.

\section{Model and Main Result}
Our Rayleigh gas model in three dimensional space is now detailed. Define $U:=[0,1]^3$ with periodic boundary conditions.
Here a tagged particle evolves via the hard sphere flow and  the remaining $N$ particles do not interact, i.e. move along straight lines. The initial distribution of the tagged particle is $f_0 \in L^1 (U \times \mathbb{R}^3 )$. The  $N$ background particles are independently distributed according to the law
$g_0 \in L^1 ( \mathbb{R}^3 )$ in velocity space and uniform in $U$.

The tagged particle and the background particles are spheres with diameter $\varepsilon >0$ which is related to $N$ via the Boltzmann-Grad scaling,
\begin{equation} \label{eq-boltzgrad}
N\varepsilon^2 =1.
\end{equation}
The background particles always travel in free flow with their velocities never changing from the initial value. The tagged particle travels in free flow whilst its centre remains at least $\varepsilon$ away from the centre of all the background particles.

When the centre of the tagged particle comes within $\varepsilon$ of the centre of a background particle the tagged particle collides as if it was a Newtonian hard-sphere collision and changes velocity.

Explicitly this is described as follows. Denote the position and velocity of background particle $1\leq j \leq N$ at time $t$ by $(x_j(t),v_j(t))$. Then for all $t\geq 0$,
\begin{align*}
 \frac{\mathrm{d}x_j(t)}{\mathrm{d}t} = v_j(t) \textrm{ and }   \frac{\mathrm{d}v_j(t)}{\mathrm{d}t}  = 0.
\end{align*}
Further denote the position and velocity of the tagged particle at time $t$ by $(x(t),v(t))$. Then for all $t\geq 0$,
\[ \frac{\mathrm{d}x(t)}{\mathrm{d}t} = v(t) .\]
If at time $t$ for all $1\leq j \leq N$, $|x(t)-x_j(t)|>\varepsilon$ then ${\mathrm{d}v(t)}/{\mathrm{d}t}=0$. Otherwise there exists a $1\leq j \leq N$ such that $|x(t)-x_j(t)|=\varepsilon$ and the tagged particle experiences an instantaneous change of velocity. Define the collision parameter $\nu \in \mathbb{S}^2$ by,
\[ \nu := \frac{x(t)-x_j(t)}{|x(t)-x_j(t)|} .\]

\begin{figure}[h!]
\centering
\includegraphics[scale=1]{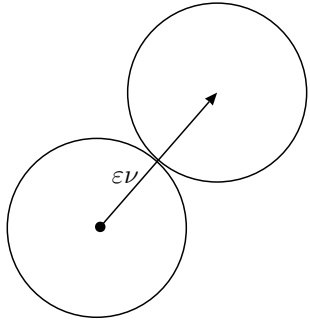}
\caption{The collision parameter $\nu$}
\label{fig-nu}
\end{figure}

Then the velocity of the tagged particle instantaneously after the collision, $v(t)$, is given by
\[ v(t):= v(t^-) - \nu \cdot (v(t^-) - v_j)\nu.   \]

 It is noted that in this model we do not have conservation of momentum. The background particles do not change velocity and the root particle does.

\begin{prop}\label{prop:wellposedHS}
For $N \in \mathbb{N}$ and $T>0$ fixed these dynamics are well defined up to time $T$ for all initial configurations apart from a set of measure zero.
\end{prop}

The proof is given in section \ref{sec:aux}, which establishes that almost surely all collisions involve only pairs.

We are interested in studying the distribution of
a tagged particle among $N$ background particles, $\hat{f}^N_t$, under the above particle dynamics as $N$ tends to infinity or equivalently as $\varepsilon$ tends to zero.
\begin{deff} \label{admissible}
Probability densities $f_0\in L^1(U\times \R^3)$, $g_0 \in L^1(\R^3)$ are admissible if
	\begin{eqnarray} 
	\label{eq-f0assmp2}
	\int_{ U \times \mathbb{R}^3} f_0 (x,v)(1+|v|^2) \, \mathrm{d}x \, \mathrm{d}v  < \infty,\\
	\label{eq-g0l1assmp}
	\int_{ \mathbb{R}^3} g_0 (v)(1+|v|^2) \, \, \mathrm{d}v  < \infty,\\
	\label{eq-g0linf5}
	\esssup_{v\in\mathbb{R}^3} g_0(v)(1+|v|^4) < \infty.
	\end{eqnarray}
\end{deff}
The distribution of the tagged particle is shown to converge to the solutions of a linear Boltzmann  equation up to a finite arbitrary time $T$. We now state the main theorem of this paper.

\begin{thm} \label{thm-main}
Let $0<T<\infty$ and $f_0,g_0$ be admissible. Then
$\hat f_t^N$ converges to a time-dependent density $f_t$ in the TV sense. Moreover, the limit
$f_t$ satisfies the linear Boltzmann equation
\begin{equation} \label{eq-linboltz}
\begin{cases}
\partial_t f_t (x,v) & = -v \cdot \partial_x f_t(x,v) + Q[f_t](x,v), \\
f_{t=0}(x,v)  & = f_0(x,v),
\end{cases}
\end{equation}
where the collision operator $Q$ is defined by $Q:=Q^+-Q^-$ and $Q^+$ and $Q^-$ known respectively as the gain and loss term are given as follows,
\[ Q^+[f](x,v) = \int_{\mathbb{S}^2} \int_{\mathbb{R}^3} f(x,v')g_0(\bar{v}')[(v-\bar{v})\cdot \nu]_+ \, \mathrm{d}\bar{v} \, \mathrm{d}\nu, \]
where the pre-collision velocities, $v'$ and $\bar{v}'$, are given by $v' = v+ \nu \cdot (\bar{v}-v) \nu$ and $\bar{v}'=\bar{v} - \nu \cdot (\bar{v}-v)\nu$, and
\[ Q^-[f](x,v) = f(x,v)\int_{\mathbb{S}^2} \int_{\mathbb{R}^3}  g_0(\bar{v})[(v-\bar{v})\cdot \nu]_+ \, \mathrm{d}\bar{v} \, \mathrm{d}\nu. \]
\end{thm}
\subsection{Remarks}
\begin{enumerate}
\item The reader is reminded that solutions of \eqref{eq-linboltz} only conserve mass, but not energy.
\item
The analysis of the Rayleigh gas can also be done using traditional BBGKY approach. Here one uses the collision
operator in $\mathbb{R}^d$:
\begin{align*}
\mathcal C^{\text{Rayl}}_{s,s+1}f^{(s+1)} (t,Z_s)&=(N-s) \varepsilon^{d-1}  \int_{S^{d-1} \times \mathbb{R}^d} \nu \cdot (v_{s+1}-v_i)\\
 & \qquad  \times f_N^{(s+1)}(t,Z_s,x_1+ \varepsilon  \nu, v_{s+1}) \, \mathrm{d}\nu \, \mathrm{d}v_{s+1}.
\end{align*}
The hard sphere collision operator is given by
\begin{align}\label{eq-Chs}
\mathcal C^{\text{hs}}_{s,s+1} f^{(s+1)}(t,Z_s)&:= (N-s) \varepsilon^{d-1} \sum_{i=1}^s \int_{S^{d-1} \times \mathbb{R}^d} \nu \cdot (v_{s+1}-v_i)
\\ \nonumber
 & \qquad  \times f_N^{(s+1)}(t,Z_s,x_i+ \varepsilon  \nu, v_{s+1}) \, \mathrm{d}\nu \, \mathrm{d}v_{s+1}.
\end{align}
The only difference between $\mathcal C^{\text{Rayl}}_{s,s+1}$ and $\mathcal C^{\text{hs}}_{s,s+1}$ is the fact that
in the hard sphere case one sums over all indices $i=1,\ldots ,s$ and in the Rayleigh case only over $i=1$. This gives estimates on $\mathcal C^{\text{Rayl}}_{s,s+1}$, which are independent of $s$ in the contraction
proof for the mild form of the associated BBGKY hierarchy. Using the function spaces $\bf{X}_{\varepsilon,\beta,\mu}$ with norm
 $ \|.\|_{\varepsilon, \beta,\mu}$  as in \cite[Def 5.1.4]{saintraymond13}
for measurable functions $G: t \in[0,T] \mapsto (t)=(g_s(t))_{s \geq 1} \in \bf{X}_{\varepsilon,\beta,\mu} $ one can introduce another time-dependent
 variant compared to \cite[Def 5.1.4]{saintraymond13}
 \[ \||G|\|_{\varepsilon, \beta,\mu,\lambda}:= \sup_{0 \leq t \leq T} \|G(t)\exp(-\lambda t |v_1|^2)  \|_{\varepsilon, \beta,\mu}. \]
This will lead to a contraction for arbitrary large times $T$. For a slightly different approach assuming only finite moments see \cite[Section II.B]{spohn80}. \\
\item Our method can be used to derive quantitative error estimates at the expense of more complex notation and additional regularity requirements for $f_0$ and $g_0$. In particular, see lemmas \ref{lem-intg0bhtbound} and \ref{lem-emp3} for some quantitative expressions.\\
\item The result should also hold in the case $d=2$ or $d \geq 4$ up to a change in moment assumptions on the initial data and minor changes in estimates and calculations throughout the paper.   \\
\item One could consider a spatially inhomogeneous initial distribution for the background particles $g_0=g_0(x,v)$. This adds a complication to the equations since for example the operator $Q$ in \eqref{eq-linboltz} becomes time-dependent, i.e. $Q_t = Q_t^+-Q^-_t$ with
\[ Q_t^-[f](x,v) = f(x,v)\int_{\mathbb{S}^2} \int_{\mathbb{R}^3}  g_0(x-t\bar{v},\bar{v})[(v-\bar{v})\cdot \nu]_+ \, \mathrm{d}\bar{v} \, \mathrm{d}\nu, \]
and $Q^+_t$ analogous.
Since the operator now depends on the time $t$ this would require evolution semigroup results to echo the semigroup results in \cite{arlotti06}.\\
\item One could also attempt to adapt these results to more complex and involved models. For example a model where each particle has an associated counter and a collision occurs between particle $i$ and $j$ if and only if both counters are less than $k$, in the hope of letting $k$ tend to infinity. The main difficulty here will be that one will need to keep track of the current distribution of the background $g_t$ in contrast to our model where the background has constant with time distribution $g_0$.
\end{enumerate}

\subsection{Method of Proof}
We closely follow the method of \cite{matt12}. That is we study the probability distribution of finding a given history of collisions, a given tree, at time $t$.

Firstly in section \ref{sec:ideal} we prove the main result, theorem \ref{thm-ideq}, which shows that there exists a solution $P_t$ to a Kolmogorov differential equation \eqref{eq-id} and relate this solution to the solution of the linear Boltzmann equation. We show existence by explicitly building a solution on the most simple trees and using this to iteratively build a full solution.

In section \ref{sec:emp} we consider the distribution $\hat{P}_t$ of finding a given history of collisions from our particle dynamics and show by direct calculation that this solves a similar differential equation in theorem~\ref{thm-emp} for sufficiently well controlled (good) trees.

Finally in section \ref{sec:conv}  we prove the main theorem of the paper, theorem~\ref{thm-main}, by proving the convergence between $P_t$ and $\hat{P}_t$ in theorem~\ref{thm-ptphatcomp} and then relating this to $f_t$ and $\hat{f}^N_t$.

\subsection{Tree Set Up}
We construct trees in a similar way to \cite{matt12}.

A tree represents a specific history of collisions. The nodes of the tree are denoted by $m$ and represent particles while the edges, denoted $E$, represent collisions. The root of the tree represents the tagged particle and is marked with the initial position of the tagged particle $(x_0,v_0)\in U \times \mathbb{R}^3$. The child nodes of the root represent background particles that the root collides with and are denoted $(t_j,\nu_j,v_j) \in (0,T]\times \mathbb{S}^2\times \mathbb{R}^3$, where $t_j$ represents the collision time, $\nu_j$ the collision parameter and $v_j$ the incoming velocity of the background particle.  Since the background particles only collide with the root particle and not each other we only consider trees of height at most 1, so the trees simplify to the initial position of the tagged particle and a list of its $n\geq 0$ collisions. The graph structure is mainly suppressed.

\begin{deff}
 The set of collision trees $\mathcal{MT}$ is defined by,
\[ \mathcal{MT}:= \{ (x_0,v_0),(t_1,\nu_1,v_1),\dots,(t_{n},\nu_{n},v_{n}) : n \in \mathbb{N}\cup \{0\}  \}. \]
For a tree $\Phi \in \mathcal{MT}$, $n(\Phi)$ is the number of collisions.

The final collision in a tree $\Phi$ plays a significant role. Define the maximum collision time $\tau(\Phi) \in [0,T]$,
\begin{equation}
\tau(\Phi):=
\begin{cases}
0 & \textrm{ if } n(\Phi) =0, \\
\max_{1\leq j \leq n(\Phi) } t_j & \textrm{ else. }
\end{cases}
\end{equation}
Further for $n(\Phi) \geq 1$ the marker for the final collision is denoted by,
\[ (\tau, \nu,v') := (t_{n(\Phi) },\nu_{n(\Phi) },v_{n(\Phi) } ). \]

\end{deff}

The realisation of a tree $\Phi$ at a time $t\in[0,T]$ for a particle diameter $\varepsilon>0$ uniquely defines the position and velocity of the tagged particle for all times up to $t$ since the initial position and the collisions the root experiences are known. Further it determines the initial positions of the $n(\Phi)$ background particles involved in the tree since we can work backwards from the collision and we know that their velocity does not change. Finally it also includes information about the other $N-n$ background particles, because it is known that they do not interfere with the root up to time $t$.

If the root collides at the instant $t$ denote the pre-collisional velocity by  $v(t^-)$ and the post-collisional velocity by $v(t)$.
Throughout this paper the dependence on $\Phi$ is often dropped from these and other variables when the context is clear.

For $n\geq 1$, define $\bar{\Phi}$ as being the pruned tree of ${\Phi}$ with the node representing the final collision, which occurs at time $\tau$, removed. For example if $\Phi = ((x_0,v_0),(t_1,\nu_1,v_1),(\tau,\nu,v'))$ then $\bar{\Phi} = ((x_0,v_0),(t_1,\nu_1,v_1))$.

\begin{figure}[ht]
\hspace*{\fill}
\subfigure[An example tree with 1 collision]
{
\includegraphics[scale=1]{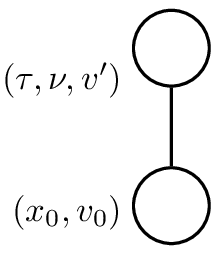}
}\hspace*{\fill}
\subfigure[An example tree with 3 collisions]
{
\includegraphics[scale=1]{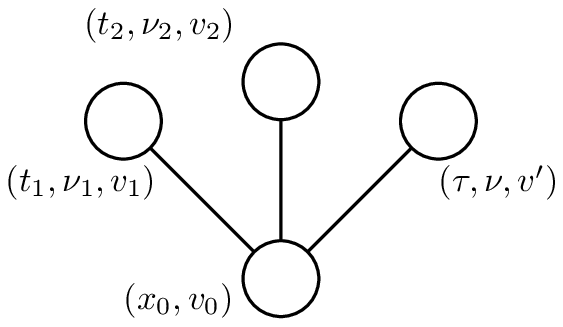}
}
\caption{Two example trees}
\hspace*{\fill}
\end{figure}
Define a metric $\mathrm{d}$ on $\mathcal{MT}$ as follows. For $\Phi,\Psi \in \mathcal{MT}$ with components $\Phi_j$  and $\Psi_j$ respectively.
\begin{equation*}
\textrm{d}(\Phi,\Psi) := \begin{cases}
1 & \textrm{ if } n(\Phi) \neq n(\Psi) \\
\min \left\{ 1,\max_{0\leq j \leq n} |\Phi_j - \Psi_j|_\infty \right\} & \textrm{ else.}
\end{cases}
\end{equation*}
Further denote by $B_h(\Phi)$ the ball of radius $h/2$ around $\Phi \in \mathcal{MT}$,
\begin{equation}\label{eq-balldeff}
 B_h(\Phi) := \{ \Psi \in \mathcal{MT} : \mathrm{d}(\Phi,\Psi) < h/2 \}.
\end{equation}
The standard Lebesgue measure on $\mathcal{MT}$ is denoted by $\mathrm{d}\lambda$.

\section{The Idealized Distribution} \label{sec:ideal}

In this section we show that there exists a solution, denoted $P_t$, to the idealized equation, equation \eqref{eq-id}, and relate this solution to the solution of the linear Boltzmann equation. We prove existence by constructing a solution iteratively on different sized trees. In section~\ref{sec:conv} $P_t$ is compared to the solution of a similar evolution equation defined by the particle dynamics in order to show the required convergence. The idealized system plays the same role as the Boltzmann hierarchy in \cite{saintraymond13}.

For a given tree $\Phi \in \mathcal{MT}$, $P_0(\Phi)$ is zero unless $\Phi$ involves no collisions, in which case $P_0(\Phi)$ is given in terms of the initial distribution $f_0$. $P_t(\Phi)$ remains zero until $t=\tau$ when there is an instantaneous increase to a positive value depending on $P_\tau(\bar{\Phi})$ and the final collision in $\Phi$. For $t>\tau$, $P_t(\Phi)$ decreases at a rate that is obtained by considering all possible collisions.

The idealized equation is given by,
\begin{equation} \label{eq-id}
\begin{cases}
\partial _t P_t (\Phi) & = \mathcal{Q}_t[P_t](\Phi) = \mathcal{Q}_t^+ [P_t](\Phi) - \mathcal{Q}^{-}_t[P_t](\Phi), \\ P_0(\Phi) & = f_0 (x_0,v_0)\mathbbm{1}_{n(\Phi)=0},
\end{cases}
\end{equation}
where,
\begin{equation}
\mathcal{Q}_t^+ [P_t](\Phi): = \left\{\begin{array}{cc}\delta(t-\tau)P_t (\bar{\Phi})g_0(v')[(v(\tau^-)-v')\cdot \nu]_+&\mbox{if } n(\Phi)\geq 1,
\\ 0&\mbox{if } n(\Phi)=0,\end{array}\right.
\end{equation}

\begin{equation}
\mathcal{Q}^{-}_t[P_t](\Phi): = P_t(\Phi) \int_{\mathbb{S}^{2}} \int_{\mathbb{R}^3} g_0(\bar{v})[(v(\tau)-\bar{v})\cdot \nu ]_+ \, \mathrm{d}\bar{v} \, \mathrm{d}\nu.
\end{equation}
\begin{thm} \label{thm-ideq}
Suppose that $f_0$ and $g_0$ are admissible (in the sense
of Def.~\ref{admissible}). Then there exists a solution $P:[0,T] \to L^1(\mathcal{MT})$ to the idealized equation, \eqref{eq-id}. Moreover for any $t\in[0,T]$ and for any $\Omega \subset U \times \mathbb{R}^3  $ define $$S_t(\Omega):=\{ \Phi \in \mathcal{MT} : (x(t),v(t)) \in \Omega \}.$$ Then
\[ \int_\Omega f_t(x,v) \, \mathrm{d}x \, \mathrm{d}v = \int_{S_t(\Omega)} P_t(\Phi) \, \mathrm{d}\Phi,\]
where $f_t$ is the unique mild solution of the linear Boltzmann equation given in proposition \ref{prop-linboltz}.
\end{thm}
\begin{rmk}
Condition \eqref{eq-g0linf5} can be relaxed to
\begin{equation}\label{eq-g0l4eta}
\esssup_{v\in\mathbb{R}^3} g_0(v)(1+|v|^{3+\eta}) < \infty
\end{equation}
for some $\eta > 0$.
\end{rmk}
From now on assume that $f_0$ and $g_0$ are admissible with the provision that either \eqref{eq-g0linf5} or \eqref{eq-g0l4eta} holds. We prove the existence by construction, taking several steps to build a solution by solving on the most simple trees first and using this solution to iteratively build a full solution.
We begin by solving the linear Boltzmann equation. We establish existence, uniqueness and regularity of solutions of \eqref{eq-linboltz} by adapting
methods from semigroup theory. The difficulty here is that after writing the linear Boltzmann equation as the sum of two unbounded operators we need to ensure that a honest semigroup is generated in order to prove existence and uniqueness. Next we adapt these semigroup techniques to define functions $P_t^{(j)}$ that describe the distribution of finding the tagged particle such that it has experienced $j$ collisions. This is key to connecting $P_t$ to the solution of the linear Boltzmann.

The following notion of mild solution suitable for transport equations is used (c.f.
\cite[Def 3.1.1]{arendt11})
\begin{deff} \label{deff-mildsol}
Consider the following system,
\begin{equation}
\begin{cases} \label{eq-mildsol}
\partial_t u(t)& = Lu(t), \\
u(0)& =u_0.
\end{cases}
\end{equation}
Where $L: D(L) \subset  L^1(U\times \mathbb{R}^3)  \to L^1(U\times \mathbb{R}^3)$ is an operator and $u_0 \in L^1(U\times \mathbb{R}^3)$ is given. The function $u:[0,T]\to U \times \mathbb{R}^3$ is called a mild solution of \eqref{eq-mildsol}  if for all $t\geq 0$,
\[ \int_0^t u(\theta)\, \mathrm{d}\theta \in D(L) \, \textrm{ and } \, L\int_0^t u(\theta)\, \mathrm{d}\theta = u(t)-u_0. \]
\end{deff}
We split the right hand side of \eqref{eq-linboltz}
into two operators, $A$ and $B$. These will appear in the construction of $P_t$.
\begin{deff} \label{deff-aoperator}
Define $D(A), D(B)\subset L^1(U \times \mathbb{R}^3)$ by,
\begin{align*}
 D(A)&:=\{ f \in L^1(U \times \mathbb{R}^3) : v\cdot \partial_x f(x,v) + Q^-[f](x,v) \in L^1(U \times \mathbb{R}^3)  \},\\
 D(B)&:= \{ f \in L^1(U\times \mathbb{R}^3) :  Q^+[f] \in  L^1(U\times \mathbb{R}^3)  \}.
\end{align*}
Then define $A:D(A) \to L^1(U \times \mathbb{R}^3)$ and $B:D(B) \to L^1(U \times \mathbb{R}^3)$ by,
\begin{align}\label{eq-adeff}
(Af)(x,v) &:= -v\cdot \partial_x f(x,v) - Q^-[f](x,v),\\
\label{eq-bdeff}
(Bf)(x,v) &:= Q^+[f](x,v).
\end{align}
\end{deff}

\begin{prop} \label{prop-linboltz}
Suppose that the assumptions in theorem~\ref{thm-ideq} hold. Then there exists a unique mild solution $f:[0,T] \to L^1(U \times \mathbb{R}^3)$ to \eqref{eq-linboltz}. Furthermore $f_t$ remains non-negative and of mass $1$, and
\begin{eqnarray}
\label{lem-ftmom}
\int_{U\times \mathbb{R}^3} f_t(x,v)(1+|v|) \, \mathrm{d}x \, \mathrm{d}v < \infty,\\
f_t\in D(B) \label{lem-ftindb}
 \end{eqnarray}
hold for all $t\in [0,T]$.
\end{prop}
\begin{proof}
	See section~\ref{sec:aux}.
\end{proof}

\begin{prop} \label{prop-p0}
There exists a unique mild solution, $P^{(0)}:[0,T] \to L^1( U \times \mathbb{R}^3)$, to the following evolution equation,
\begin{equation} \label{eq-p0}
\begin{cases}
\partial_t P^{(0)}_t (x,v) & = (AP^{(0)}_t)(x,v), \\
 P^{(0)}_0 (x,v) & = f_0(x,v).
\end{cases}
\end{equation}
Where $A$ is as in \eqref{eq-adeff}.
\end{prop}
The distribution $P^{(0)}_t(x,v)$ can be thought of as the probability of finding the tagged particle at $(x,v)$ at time $t$ such that it has not yet experienced any collisions.
\begin{proof}
By lemma~\ref{lem-tsemi} $A$ generates the substochastic $C_0$ semigroup $T(t)$ given in \eqref{eq-tsemi}. By the Hille-Yoshida theorem,  \cite[Thm 1.3.1]{pazy83} $A$ is closed. By  \cite[Thm 3.1.12]{arendt11} \eqref{eq-p0} has a unique mild solution given by
\begin{equation} \label{eq-pt0ttf0}
P_t^{(0)}=T(t)f_0.
\end{equation}
\end{proof}
\begin{lem} \label{lem-p0leqft}
For all $t \in [0,T]$, $P_t^{(0)} \leq f_t$ pointwise.
\end{lem}
\begin{rmk}
This lemma is entirely expected. The probability of finding the tagged particle at $(x,v)$ at time $t$ is given by $f_t(x,v)$ and the probability of finding it at $(x,v)$ at time $t$ such that it has not experienced any collisions up to time $t$ is given by $P_t^{(0)}(x,v)$ so one expects $P_t^{(0)} \leq f_t$.
\end{rmk}
\begin{proof}
For $t \in [0,T]$ define $F_t^{(0)}:=f_t-P_t^{(0)}$. Then since $f_t$ and $P_t^{(0)}$ are mild solutions of \eqref{eq-linboltz} and \eqref{eq-p0} respectively, $F_t^{(0)}$ is a mild solution of
\begin{equation*}
\begin{cases}
\partial_t F_t^{(0)}(x,v) & = AF_t^{(0)} + Bf_t(x,v) \\
F_0^{(0)}(x,v) & = 0.
\end{cases}
\end{equation*}
By \eqref{lem-ftindb} and \cite[Prop. 3.1.16]{arendt11} $F_t^{(0)}$ is given by,
\[ F_t^{(0)} = \int_0^t T(t-\theta)Bf_\theta \, \mathrm{d}\theta.  \]
Now noting that $f_\theta$ is non-negative it follows that $Bf_\theta$ and hence $T(t-\theta)Bf_\theta$ are non-negative also. Hence $F_t^{(0)} \geq 0$ which implies $P_t^{(0)} \leq f_t$.
\end{proof}

\begin{deff} \label{deff-ptfull}
For $j \in \mathbb{N} \cup \{0\}$ denote by $\mathcal{T}_j$ the set of all trees with exactly $j$ collisions. Explicitly,
\begin{equation} \label{eq-tjdeff}
\mathcal{T}_j := \{ \Phi \in \mathcal{MT} : n(\Phi) = j \}.
\end{equation}
The required solution $P_t$ can now be defined iteratively on the space $\mathcal{MT}$. For $\Phi \in \mathcal{T}_0$ define
\begin{equation} \label{eq-ptt0}
 P_t(\Phi) := P^{(0)}_t (x(t),v(t)).
\end{equation}
Else define,
\begin{align} \label{eq-ptfull}
P_t(\Phi): = \mathbbm{1}_{t \geq \tau} \exp \left( -(t-\tau) \int_{\mathbb{S}^2} \int_{\mathbb{R}^3} g_0(\bar{v}) [(v(\tau)-\bar{v})\cdot \nu']_+ \, \mathrm{d}\bar{v} \, \mathrm{d}\nu'  \right) \nonumber \\ P_\tau (\bar{\Phi})g_0(v')[(v(\tau^-)-v')\cdot \nu]_+.
\end{align}
The right hand side of this equation depends on $P_\tau (\bar{\Phi})$ but since $\bar{\Phi}$ has degree exactly one less than $\Phi$ the equation is well defined.
\end{deff}

The proof that $P_t$ has the required properties of theorem~\ref{thm-ideq} is given shortly. We first define the function $P^{(j)}_t$ which is thought of, in parallel to $P^{(0)}_t$, as being the probability of finding the tagged particle at a certain position at time $t$ such that it has experienced exactly $j$ collisions up to time $t$. The $P_t^{(j)}$ will be required to show the connection between $P_t$ and the solution of the linear Boltzmann equation.
\begin{deff}\label{def-pj}
Let $t\in [0,T]$ and $\Omega \subset U \times \mathbb{R}^3$ be measurable. Recall in theorem \ref{thm-ideq} we define the set, $S_t(\Omega)=\{ \Phi \in \mathcal{MT} : (x(t),v(t))\in \Omega \}$ - the set of all trees such that the tagged particle at time $t$ is in $\Omega$. Define for all $j \in \mathbb{N}\cup \{0\}$,
\[ S^j_t(\Omega):=\mathcal{T}_j \cap S_t(\Omega) . \]
Then for $j \geq 1$, define $P^{(j)}(\Omega)$ by,
\[ P^{(j)}_t(\Omega):= \int_{ S^j_t(\Omega) } P_t(\Phi) \, \mathrm{d}\Phi. \]
\end{deff}
\begin{lem} \label{lem-Ptjmble}
Let $t\in [0,T]$, $j \geq 1$. Then $P_t^{(j)}$ is absolutely continuous with respect to the Lebesgue measure on $U \times \mathbb{R}^3$.
\end{lem}
\begin{proof}
Let $j=1$. Then we have by \eqref{eq-ptfull},
\begin{align} \label{eq-ptjomegaform}
P^{(1)}_t(\Omega)& = \int_{ S^1_t(\Omega) } P_t(\Phi) \nonumber \\
& =  \int_0^t \int_{\mathbb{R}^3} \int_{\mathbb{R}^3} \int_{\mathbb{S}^2} \int_U \exp \left( -(t-\tau) \int_{\mathbb{S}^2} \int_{\mathbb{R}^3} g_0(\bar{v}) [(v(\tau)-\bar{v})\cdot \nu]_+ \, \mathrm{d}\bar{v} \, \mathrm{d}\nu  \right) \nonumber \\
& \qquad \qquad P_\tau^{(0)}(x_0+ \tau v_0,v_0)g_0(v') [(v_0-v')\cdot \nu]_+ \mathbbm{1}_{(x(t),v(t)) \in \Omega)} \, \mathrm{d}x_0 \, \mathrm{d}\nu \, \mathrm{d}v' \, \mathrm{d}v_0   \, \mathrm{d}\tau
\end{align}
We define a coordinate transform $(\nu, x_0,v_0,v') \mapsto (\nu, x,v,\bar{w})$ given by,
\begin{align*}
v & :=v_0 + \nu (v'-v_0)\cdot \nu  \\
x & := x_0 + \tau v_0 + (t-\tau) v \\
\bar{w} & := v' - \nu (v'-v_0)\cdot \nu.
\end{align*}
This transformation has Jacobi matrix,
\[  \begin{pmatrix}
\textrm{Id} & 0 & 0 & 0 \\
 & \textrm{Id} & &  \\
 & 0 & \textrm{Id}-\nu \otimes \nu & \nu \otimes \nu \\
 & 0 & \nu \otimes \nu & \textrm{Id}-\nu \otimes \nu \\
\end{pmatrix} \]
where the blank entries are not required for the computation of the matrix's determinant. The 2x2 matrix in the bottom right has determinant $- 1$ and hence the absolute value of the determinant of the entire matrix is $1$. With this \eqref{eq-ptjomegaform} becomes,
\begin{align*}
P^{(1)}_t(\Omega)& = \int_{\Omega} \int_0^t \int_{\mathbb{R}^3}  \int_{\mathbb{S}^2}  \exp \left( -(t-\tau) \int_{\mathbb{S}^2} \int_{\mathbb{R}^3} g_0(\bar{v}) [(v-\bar{v})\cdot \nu]_+ \, \mathrm{d}\bar{v} \, \mathrm{d}\nu  \right) \nonumber \\
& \qquad \qquad P_\tau^{(0)}(x-(t-\tau)v,w')g_0(\bar{w}') [(v-\bar{w})\cdot \nu]_+  \, \mathrm{d}\nu \, \mathrm{d}\bar{w} \,  \mathrm{d}\tau \, \mathrm{d}x \, \mathrm{d}v \nonumber,
\end{align*}
where $w'=v + \nu (\bar{w} - v) \cdot \nu$ and $\bar{w}'=\bar{w} - \nu (\bar{w} - v) \cdot \nu$. Hence we see that if the Lebesgue measure of $\Omega$ equals zero then so does $P_t^{(1)}(\Omega)$. For $j \geq 2$ we use a similar approach, using the iterative formula for $P_t(\Phi)$ \eqref{eq-ptfull}.
\end{proof}
\begin{rmk}
Since $P_t^{(j)}$ is an absolutely continuous measure on $U\times \mathbb{R}^3$, the Radon-Nikodym theorem (see Theorem 4.2.2 \cite{cohn13}) implies that $P_t^{(j)}$ has a density, which we denote by $P_t^{(j)}$ also. This gives,
\[ \int_{\Omega} P_t^{(j)}(x,v) \, \mathrm{d}x \, \mathrm{d}v = \int_{ S^j_t(\Omega) } P_t(\Phi) \, \mathrm{d}\Phi.  \]
Hence for almost all $(x,v) \in U \times \mathbb{R}^3$,
\[   P_t^{(j)}(x,v)  = \int_{S^j_t(x,v) } P_t(\Phi) \, \mathrm{d}\Phi. \]
\end{rmk}

\begin{rmk} \label{rmk-p0pt}
A similar formula holds for $P^{(0)}_t$ since the set $S^0_t(x,v)$ contains exactly one tree: the tree $\Phi$ with initial root data $(x-tv,v)$ and such that the root has no collisions,
\begin{align*}
\int_{S^0_t(x,v) } P_t(\Phi) \, \mathrm{d}\Phi &= P_t((x-tv,v)) = P^{(0)}_t(x,v).
\end{align*}
\end{rmk}

\begin{prop} \label{prop-pj}
For $j\geq 1$, $P^{(j)}_t$ as defined above is almost everywhere the unique mild solution to the following differential equation,
\begin{equation*}
\begin{cases}
\partial_t P^{(j)}_t(x,v) &  = AP_t^{(j)}(x,v) + BP_t^{(j-1)}(x,v), \\
P^{(j)}_0(x,v)& =0,
\end{cases}
\end{equation*}
where $A$ is given in \eqref{eq-adeff} and $B$ in \eqref{eq-bdeff}.
\end{prop}

The following lemma helps prove the proposition for the case $j=1$ which allows the use of an inductive argument to prove the proposition in full.
\begin{lem} \label{lem-p1alt}
For any $t \in [0,T]$ and almost all $(x,v)\in U \times \mathbb{R}^3$,
\begin{equation*}
P^{(1)}_t(x,v)= \int_0^t T(t-\theta)BP_\theta^{(0)}(x,v) \, \mathrm{d}\theta,
\end{equation*}
where the semigroup $T(t)$ is as in \eqref{eq-tsemi}.  The right hand side is well defined since  \eqref{lem-ftindb}, \eqref{lem-p0leqft} and \eqref{eq-bfl1bound} imply $P_t^{(0)} \in D(B)$.
\end{lem}
\begin{proof}
We show that for any $\Omega \subset U \times \mathbb{R}^3$ measurable we have,
\begin{equation*} \label{eq-p1alteq}
 \int_{ \Omega } P_t^{(1)}(x,v) \, \mathrm{d}x \, \mathrm{d}v = \int_{\Omega} \int_0^t T(t-\theta)BP_\theta^{(0)}(x,v) \, \mathrm{d}\theta \, \mathrm{d}x\, \mathrm{d}v. \\
\end{equation*}
By the definition of $T(t)$ in equation \eqref{eq-tsemi}, the definition of $B$ in \eqref{eq-bdeff}, and the proof of lemma~\ref{lem-Ptjmble}, for $w'=v + \nu (\bar{w} - v) \cdot \nu$ and $\bar{w}'=\bar{w} - \nu (\bar{w} - v) \cdot \nu$, we have,
\begin{align*} \label{eq-p1altrhs}
\int_{\Omega} \int_0^t & T(t-\theta)BP_\theta^{(0)}(x,v) \, \mathrm{d}\theta \, \mathrm{d}x\, \mathrm{d}v  \\
& = \int_{\Omega}   \int_0^t \exp \left( -(t-\theta) \int_{\mathbb{S}^2} \int_{\mathbb{R}^3} g_0(\bar{v}) [(v-\bar{v})\cdot \nu]_+ \, \mathrm{d}\bar{v} \, \mathrm{d}\nu  \right) \nonumber \\
& \qquad \qquad \times BP_\theta^{(0)}(x-(t-\theta)v,v)  \, \mathrm{d}\theta \, \mathrm{d}x\, \mathrm{d}v \\
& = \int_{\Omega}   \int_0^t \exp \left( -(t-\theta) \int_{\mathbb{S}^2} \int_{\mathbb{R}^3} g_0(\bar{v}) [(v-\bar{v})\cdot \nu]_+ \, \mathrm{d}\bar{v} \, \mathrm{d}\nu  \right) \nonumber \\
& \qquad \int_{\mathbb{R}^3}  \int_{\mathbb{S}^2} P_\theta^{(0)}(x-(t-\theta)v,w')g_0(\bar{w}') [(v-\bar{w})\cdot \nu]_+  \, \mathrm{d}\nu \, \mathrm{d}\bar{w} \,  \mathrm{d}\theta \, \mathrm{d}x \, \mathrm{d}v \nonumber \\
& = \int_{\Omega} \int_0^t \int_{\mathbb{R}^3}  \int_{\mathbb{S}^2}  \exp \left( -(t-\theta) \int_{\mathbb{S}^2} \int_{\mathbb{R}^3} g_0(\bar{v}) [(v-\bar{v})\cdot \nu]_+ \, \mathrm{d}\bar{v} \, \mathrm{d}\nu  \right) \nonumber \\
& \qquad \qquad P_\theta^{(0)}(x-(t-\theta)v,w')g_0(\bar{w}') [(v-\bar{w})\cdot \nu]_+  \, \mathrm{d}\nu \, \mathrm{d}\bar{w} \,  \mathrm{d}\theta \, \mathrm{d}x \, \mathrm{d}v \nonumber \\
& = \int_{ \Omega }  P_t^{(1)}(x,v) \, \mathrm{d}x \, \mathrm{d}v.
\end{align*}

\end{proof}
\begin{proof}[Proof of proposition~\ref{prop-pj}]
Consider induction on $j$. First let $j=1$. We seek to apply \cite[Prop 3.1.16]{arendt11}. If $\int_0^t BP_\theta^{(0)} \, \mathrm{d}\theta \in L^1(U\times \mathbb{R}^3)$ then the proposition holds so by the above lemma $P_t^{(1)}$ is the unique mild solution.

To this aim note that since $P_t^{(0)}$ is the unique mild solution of \eqref{eq-p0}, \[ \int_0^t P_\theta^{(0)} \, \mathrm{d}\theta \in D(A). \]
By \cite[Section 10.4.3]{arlotti06} $D(A) \subset D(B)$ and hence
\[ \int_0^t P_\theta^{(0)} \, \mathrm{d}\theta \in D(B). \]
This implies,
\[ B\int_0^t P_\theta^{(0)} \, \mathrm{d}\theta \in L^1(U\times \mathbb{R}^3). \]
as required.
Now consider $j\geq 2$ and assume the proposition is true for $j-1$. By setting $F_t^{(j-1)}:=f_t-P_t^{(j-1)}$ a similar argument to lemma~\ref{lem-p0leqft} shows that $P_t^{(j-1)} \leq f_t$. By \eqref{lem-ftindb} and \eqref{eq-bfl1bound}, $P_t^{(j-1)} \in D(B)$ so the right hand side is well defined.
A similar approach to lemma~\ref{lem-p1alt} shows that for any $t \in [0,T]$  and almost all $(x,v)\in U \times \mathbb{R}^3$,
\begin{equation*}
P^{(j)}_t(x,v)= \int_0^t (T(t-\theta)BP_\theta^{(j-1)})(x,v) \, \mathrm{d}\theta,
\end{equation*}
where $T(t)$ is the semigroup given in \eqref{eq-tsemi}. The rest follows by the same argument as in the $j=1$ case.
\end{proof}

\begin{prop} \label{prop-pjsum}
For all $t \in [0,T]$ and almost all $(x,v) \in U \times \mathbb{R}^3$,
\begin{equation} \label{eq-pjsum}
\sum_{j=0}^\infty P^{(j)}_t(x,v) = f_t(x,v),
\end{equation}
where $f_t$ is the unique mild solution of the linear Boltzmann equation given in proposition~\ref{prop-linboltz}.
\end{prop}
\begin{proof}
Since $P_t^{(0)}$ is a mild solution of \eqref{eq-p0},
\begin{equation} \label{eq-p0mildsolda}
\int_0^t P_\theta^{(0)}(x,v) \, \mathrm{d}\theta \in D(A),
\end{equation}
and,
\begin{equation} \label{eq-p0mildsol}
P_t^{(0)}(x,v)  = f_0(x,v) + A\int_0^t P_\theta^{(0)}(x,v) \, \mathrm{d}\theta.
\end{equation}
Further for $j\geq 1$ by proposition~\ref{prop-pj} and \cite[Prop 3.31]{arlotti06},
\begin{equation} \label{eq-pjmildsolda}
\int_0^t P_\theta^{(j)}(x,v) \, \mathrm{d}\theta \in D(A),
\end{equation}
and,
\begin{equation} \label{eq-pjmildsol}
P_t^{(j)}(x,v) = A\int_0^t P_\theta^{(j)}(x,v) \, \mathrm{d}\theta  + \int_0^t BP_\theta^{(j-1)}(x,v) \, \mathrm{d}\theta.
\end{equation}
Combining \eqref{eq-p0mildsolda} and \eqref{eq-pjmildsolda},
\[ \int_0^t \sum_{j=0}^\infty P_\theta^{(j)}(x,v) \, \mathrm{d}\theta \in D(A). \]
Recalling from the proof of proposition~\ref{prop-pj} that $D(A) \subset D(B)$,
\[ \int_0^t \sum_{j=0}^\infty P_\theta^{(j)}(x,v) \, \mathrm{d}\theta \in D(A) \cap D(B)=D(A+B). \]
Further summing \eqref{eq-p0mildsol} and \eqref{eq-pjmildsol} for $j \geq 1$,
\begin{align*}
\sum_{j=0}^\infty P_t^{(j)}(x,v) & =  f_0(x,v) +  \sum_{j=0}^\infty A\int_0^t P_\theta^{(j)}(x,v) \, \mathrm{d}\theta + \sum_{j=1}^\infty \int_0^t BP_\theta^{(j-1)}(x,v) \, \mathrm{d}\theta \\
& = f_0(x,v) + A \int_0^t \sum_{j=0}^\infty  P_\theta^{(j)}(x,v) \, \mathrm{d}\theta + \int_0^t B \sum_{j=0}^\infty P_\theta^{(j)}(x,v) \, \mathrm{d}\theta \\
& = f_0(x,v) + (A+B)\int_0^t \sum_{j=0}^\infty  P_\theta^{(j)}(x,v).
\end{align*}
Hence by definition~\ref{deff-mildsol}, $\sum_{j=0}^\infty  P_t^{(j)}(x,v)$ is a mild solution of \eqref{eq-linboltzoperatorform} and therefore since $f_t$ is the unique mild solution the proof is complete.
\end{proof}
We now have all the results needed to prove that $P_t$ satisfies all the requirements of theorem \ref{thm-ideq}.

\begin{proof}[Proof of theorem \ref{thm-ideq}]
Using definition~\ref{def-pj}, proposition \ref{prop-pjsum}, and, since each $P_t^{(j)}$ is positive, the monotone convergence theorem we have for any measurable $\Omega \subset U \times \mathbb{R}^3$,
\begin{align} \label{eq-ptftconn}
\int_{S_t(\Omega)} P_t(\Phi) \, \mathrm{d}\Phi & = \sum_{j=0}^\infty \int_{S_t^{j}(\Omega)} P_t(\Phi) \, \mathrm{d}\Phi
 = \sum_{j=0}^\infty \int_{\Omega} P^{(j)}_t(x,v) \, \mathrm{d}x \, \mathrm{d}v \nonumber \\
& =  \int_{\Omega} f_t(x,v) \, \mathrm{d}x \, \mathrm{d}v.
\end{align}
In particular,
\[ \int_{\mathcal{MT}} P_t(\Phi) \, \mathrm{d}\Phi =  \int_{U \times \mathbb{R}^3} f_t(x,v) \, \mathrm{d}x \, \mathrm{d}v < \infty. \]
Hence $P_t \in L^1(\mathcal{MT})$. To show that $P_t$ is a solution of \eqref{eq-id} first consider $\Phi \in \mathcal{T}_0$. Since $n(\Phi)=0$,
\[ P_0(\Phi) = P^{(0)}_0 (x(0),v(0)) = f_0(x_0,v_0) =  f_0(x_0,v_0)\mathbbm{1}_{n(\Phi)=0}. \]
Hence it solves the initial condition. Now for $t>0$, since $\Phi \in \mathcal{T}_0$, $v(t)=v_0$ and $x(t)=x_0+tv_0$. Hence by \eqref{eq-tsemi} and \eqref{eq-pt0ttf0},
\begin{align*}
P_t(\Phi) & = P_t^{(0)}(x(t),v(t))  = P_t^{(0)}(x_0+tv_0,v_0) \\
& = \exp \left( -t \int_{\mathbb{S}^2} \int_{\mathbb{R}^3} g_0(\bar{v}) [(v_0-\bar{v})\cdot \nu]_+ \, \mathrm{d}\bar{v} \, \mathrm{d}\nu  \right) f_0(x_0,v_0).
\end{align*}
The only dependence on $t$ here is in the exponential term so we differentiate $P_t(\Phi)$ with respect to $t$,
\begin{align*}
\partial_t P_t(\Phi) & = \partial_t \left( \exp \left( -t \int_{\mathbb{S}^2} \int_{\mathbb{R}^3} g_0(\bar{v}) [(v_0-\bar{v})\cdot \nu]_+ \, \mathrm{d}\bar{v} \, \mathrm{d}\nu  \right) f_0(x_0,v_0)  \right) \\
& = -\int_{\mathbb{S}^2} \int_{\mathbb{R}^3} g_0(\bar{v}) [(v_0-\bar{v})\cdot \nu]_+ \, \mathrm{d}\bar{v} \, \mathrm{d}\nu \times P_t(\Phi) = -\mathcal{Q}_t^-[P_t](\Phi).
\end{align*}
Hence $P_t$ solves \eqref{eq-id} on $\mathcal{T}_0$.

We now consider $ \Phi \in \mathcal{T}_j$ for $j \geq 1$. Since $\Phi \in \mathcal{T}_j$ we have $n(\Phi)=j$ and $\tau >0$. Hence
\begin{equation} \label{eq-pt01}
P_0(\Phi) =  0 =  f_0(x_0,v_0)\mathbbm{1}_{n(\Phi)=0}.
\end{equation}
For $t=\tau$,
\begin{align} \label{eq-pt02}
P_\tau(\Phi) & =  P_\tau(\bar{\Phi})g_0(v')[v(\tau^-)_-v')\cdot \nu]_+.
\end{align}
Further for $t>\tau$ the only dependence on $t$ is inside the exponential term and hence differentiating gives,
\begin{align} \label{eq-pt03}
\partial_t P_t(\Phi) & = \partial_t \Big( \exp \big(-(t-\tau) \int_{\mathbb{S}^2} \int_{\mathbb{R}^3} g_0(\bar{v}) [(v(\tau)-\bar{v})\cdot \nu']_+ \, \mathrm{d}\bar{v} \, \mathrm{d}\nu'  \big) \nonumber \\
& \qquad \qquad  P_\tau (\bar{\Phi})g_0(v')[(v_0-v')\cdot \nu]_+ \Big) \nonumber \\
& =   \exp \left( -(t-\tau) \int_{\mathbb{S}^2} \int_{\mathbb{R}^3} g_0(\bar{v}) [(v(\tau)-\bar{v})\cdot \nu']_+ \, \mathrm{d}\bar{v} \, \mathrm{d}\nu'  \right)\nonumber \\
& \qquad \qquad  P_\tau (\bar{\Phi})g_0(v')[(v_0-v')\cdot \nu]_+  \nonumber \\
& \qquad \qquad  \times \left( - \int_{\mathbb{S}^2} \int_{\mathbb{R}^3} g_0(\bar{v}) [(v(\tau)-\bar{v})\cdot \nu']_+ \, \mathrm{d}\bar{v} \, \mathrm{d}\nu' \right) \nonumber \\
& = - P_t(\Phi) \int_{\mathbb{S}^2} \int_{\mathbb{R}^3} g_0(\bar{v}) [(v(\tau)-\bar{v})\cdot \nu']_+ \, \mathrm{d}\bar{v} \, \mathrm{d}\nu'= -\mathcal{Q}^-_t[P_t](\Phi).
\end{align}
Equations \eqref{eq-pt01}, \eqref{eq-pt02} and \eqref{eq-pt03} prove that $P_t$ solves \eqref{eq-id} on $\mathcal{T}_j$. Since $\mathcal{MT}$ is the disjoint union of $\mathcal{T}_j$ for $j \geq 0$, $P_t$ is a solution of \eqref{eq-id} on $\mathcal{MT}$.
Finally, the required connection between $P_t$ and the solution of the linear Boltzmann equation has been shown in \eqref{eq-ptftconn}.
\end{proof}

\section{The Empirical Distribution}\label{sec:emp}
We now consider the empirical distribution of trees $\hat{P}_t^\varepsilon$ defined by the dynamics of the particle system for particles with diameter $\varepsilon$. To ease notation we drop the dependence on $\varepsilon$ and write $\hat{P}_t$. The key result of this section is that $\hat{P}_t$ solves the differential equation \eqref{eq-emp} which is similar to the idealized equation \eqref{eq-id}. The similarity between the two equations is exploited in the next section to prove the required convergence as $\varepsilon$ tends to zero.

We do this by restricting our attention to trees that are well controlled in various ways, calling these trees good trees.
\begin{deff}
For a tree $\Phi \in \mathcal{MT}$ define $ \mathcal{V}(\Phi) \in [0,\infty)$ to be the maximum velocity involved in the tree. That is,
\[ \mathcal{V}(\Phi):= \max \left\{ \max_{j=1,\dots,n(\Phi)}\{|v_j|\}, \max_{s \in [0,T]}\{|v(s)| \} \right\}.  \]
\end{deff}
\begin{deff} \label{def-recollisionfree}
A tree $\Phi \in \mathcal{MT}$ is called re-collision free at diameter $\varepsilon$ if for all $0 \leq \varepsilon'\leq \varepsilon$, for all $1 \leq j \leq n(\Phi)$ and for all $t < t_j$,
\[ |x(t)-(x_j+tv_j)|>\varepsilon'. \]
That is to say, if the tree involves a collision between the root particle and background particle $j$ at time $t_j$ then the root particle has not previously collided with background particle $j$. So if a tree is re-collision free then it involves at most one collision per background particle.

Further define
\[   R(\varepsilon)  : = \{ \Phi \in \mathcal{MT} : \Phi \textrm{ is re-collision free at diameter } \varepsilon   \}. \]
\end{deff}
\begin{deff}
A tree $\Phi \in \mathcal{MT}$ is called non-grazing if all collisions in $\Phi$ are non-grazing, that is if,
\[  \min_{1\leq j\leq n(\Phi)} \nu_j \cdot (v(t_j^-)-v_j) >0. \]
\end{deff}
\begin{deff}
A tree is $\Phi \in \mathcal{MT}$ is called free from initial overlap at diameter $\varepsilon>0$ if initially the root is at least $\varepsilon$ away from the centre of each background particle. Explicitly if, for $j=1,\dots,N$,
\[ |x_0-x_j|>\varepsilon. \]
Define $S(\varepsilon) \subset \mathcal{MT}$ to be the set of all trees that are free from initial overlap at radius $\varepsilon$.
\end{deff}
\begin{deff} \label{def-goodtrees}
For any pair of decreasing functions $V,M:(0,\infty) \to [0,\infty)$ such that $\lim_{\varepsilon \to 0} V(\varepsilon) = \lim_{\varepsilon \to 0}M(\varepsilon) =\infty $, the set of good trees of diameter $\varepsilon$, $\mathcal{G}(\varepsilon)$, is defined as,
\begin{align*}
\mathcal{G}(\varepsilon) &: = \Big\{   \Phi \in \mathcal{MT} : n(\Phi) \leq M(\varepsilon), \, \mathcal{V}(\Phi) \leq V(\varepsilon), \\ & \qquad \qquad \Phi \in  R(\varepsilon) \cap S(\varepsilon) \, \textrm{ and } \Phi \textrm{ is non-grazing}   \Big\}
\end{align*}
Since $M,V$ are decreasing for $\varepsilon' < \varepsilon$ we have $\mathcal{G}(\varepsilon) \subset \mathcal{G}(\varepsilon')$. Later some conditions on $M$ and $V$ are required to prove that $\hat{P}_t$ solves the relevant equation and to prove convergence.
\end{deff}
Now define the operator $\hat{\mathcal{Q}}_t$ which mirrors the idealized operator $\mathcal{Q}_t$ in the empirical case. Fix $\hat{C}_1>0$, a constant depending only on $\Phi$ described later. Define the gain operator,
\[ \hat{\mathcal{Q}}_t^+ [\hat{P}_t](\Phi): =\left\{ \begin{array}{ll} \delta (t - \tau) \hat{P}_t(\bar{\Phi}) \frac{ g_0(v')[(v(\tau^-)-v')\cdot \nu ]_+}{1- \pi \varepsilon^2 \int_0^\tau |v(s)-v'| \, \mathrm{d}s + \hat{C}_1\varepsilon^3}& \mbox{if } n\geq 1\\0& \mbox{if } n = 0.\end{array}\right.
\]
Next for a given tree $\Phi$, a time $0<t<T$ and $\varepsilon>0$, define the function $\mathbbm{1}_t^\varepsilon[\Phi]:  U \times \mathbb{R}^3 \rightarrow \{0,1\}$ by
\begin{equation} \label{eq-deff1phi}
 \mathbbm{1}_t^\varepsilon[\Phi] (\bar{x},\bar{v}):= \begin{cases}
1 \textrm{ if for all } s \in (0,t), \, |x(s)-(\bar{x}+s\bar{v})| > \varepsilon, \\ 0 \textrm{ else}.
\end{cases}
\end{equation}

That is $\mathbbm{1}_t^\varepsilon[\Phi] (\bar{x},\bar{v})$ is $1$ if a background particle starting at the position $(\bar{x},\bar{v})$ avoids colliding with the root particle of the tree $\Phi$ up to the time $t$. This allows us to define the loss operator,
\[ \hat{\mathcal{Q}}_t^- [\hat{P}_t](\Phi): =  \hat{P}_t(\Phi)\frac{\int_{\mathbb{S}^{2}}  \int_{\mathbb{R}^3}   g_0(\bar{v})[(v(\tau)-\bar{v})\cdot \nu]_+ \, \mathrm{d}\bar{v} \, \mathrm{d}\nu - \hat{C}_2(\varepsilon)}{\int_{ U \times {R}^3}  g_0(\bar{v}) \mathbbm{1}_t^\varepsilon[\Phi] (\bar{x},\bar{v}) \, \mathrm{d}\bar{x} \, \mathrm{d}\bar{v}  }. \]
For some $\hat{C}_2(\varepsilon)>0$ depending on $t$ and $\Phi$ of $o(1)$ as $\varepsilon$ tends to zero. Finally define the operator $\hat{\mathcal{Q}}_t$ as follows,
\[ \hat{\mathcal{Q}}_t = \hat{\mathcal{Q}}_t^+ - \hat{\mathcal{Q}}_t^-  . \]
\begin{thm} \label{thm-emp}
For $\varepsilon$ sufficiently small and for $\Phi \in \mathcal{G}(\varepsilon)$, $\hat{P}_t$ solves the following
\begin{equation} \label{eq-emp}
\begin{cases}
\partial_t \hat{P}_t(\Phi) & = (1-\gamma(t)) \hat{\mathcal{Q}}_t [\hat{P_t}](\Phi) \\
\hat{P}_0 (\Phi)& = \zeta(\varepsilon)  f_0 (x_0,v_0) \mathbbm{1}_{n(\Phi) = 0}.
\end{cases}
\end{equation}
The functions $\gamma$ and $\zeta$ are given by
\begin{equation} \label{eq-zetadeff}
\zeta(\varepsilon) := (1- \frac{4}{3}\pi \varepsilon^3)^N,
\end{equation}
and,
\[ \gamma(t) :=  \left\{ \begin{array}{ll} n(\bar{\Phi})\varepsilon^{2}&\mbox{if } t=\tau,\\
  n(\Phi)\varepsilon^2& \mbox{if } t>\tau. \end{array}\right.\]
\end{thm}
\begin{rmk} When choosing  the background particles according to some Poisson point process some of  these terms simplify as in \cite{matthies10}.
\end{rmk}

The proof is developed by a series of lemmas in which we prove the gain term, loss term and initial condition separately.
\begin{deff} \label{deff-omega}
Define $\omega_0:=(u_0,w_0) \in U \times \mathbb{R}^3$ to be the random initial position of the test particle. By our model $\omega_0$ has distribution $f_0$.

Further for $j=1,\dots,N$ define $\omega_j:=(u_j,w_j)$ to be the random  initial position and velocity of background particle $j$. Note that $\omega_j$ has distribution $\textrm{Unif}(U) \times g_0$. Finally define $\omega := (\omega_1,\dots,\omega_N)$.
\end{deff}
\begin{lem} \label{lem-phatabscts}
Let $\varepsilon>0$ and $\Psi \in \mathcal{G}(\varepsilon)$ then $\hat{P}_t$ is absolutely continuous with respect to the Lebesgue measure $\lambda$ on a neighbourhood of $\Psi$.
\end{lem}
\begin{proof}
Recall the definition of $B_h(\Psi)$ \eqref{eq-balldeff}. Since $\mathcal{G}(\varepsilon)$ is open there exists a $h>0$ such that $B_h(\Psi) \subset \mathcal{G}(\varepsilon)$. In the case $n(\Psi) =0$ for all $t \geq 0$,  $\hat{P}_t(\Psi) \leq f_0(x_0,v_0)$ and hence absolute continuity follows.

Suppose $n(\Phi) \geq 1$. Define a map $\varphi:B_h(\Psi) \to \mathcal{MT} \times U\times \mathbb{R}^3$,
\[ \varphi(\Phi) := (\bar{\Phi},(x(\tau)+\varepsilon \nu - \tau v',v')). \]
We view $\varphi$ as having $n(\Phi)+1$ components, the first being the initial root position $(x_0,v_0)$, components $j=2\dots,n$ being the marker $(t_j,\nu_j,\nu_j)$ and the final component being $(x(\tau)+\varepsilon \nu - \tau v',v')$ - the initial position of the background particle that leads to the final collision with the root in $\Phi$. We claim that,
\begin{equation} \label{eq-detgradphi}
\det(\nabla \varphi)(\Phi) = \varepsilon^2 (v(\tau^-) - v')\cdot \nu.
\end{equation}
To prove this we first rotate our coordinate axis so that $\nu = \textbf{e}_1$. Then for $k=0,\dots,n$ define $F_{0,k}: = \nabla_{x_0} \varphi_k(\Phi) $ and for $j=1,\dots, n$ define $F_{j,k}: = \nabla_{t_k,\nu_k}\varphi_j(\Phi)$.
We calculate,
\[ F_{n+1,n+1} = \nabla_{\tau,\nu}\varphi_{n+1}(\Phi) = \begin{pmatrix}
(v(\tau^-)-v')\cdot \nu & 0 & 0 \\
 & \varepsilon & 0 \\
 & 0 & \varepsilon \\
\end{pmatrix}, \]
where the blank components are not needed. Also,
$ F_{0,n} = \nabla_{x_0} \varphi_{n}(\Phi) = \textrm{Id}(2). $ Further for $j=2,\dots,n+1$,
$F_{j,j} = \nabla_{t_j,\nu_j}\varphi_{j}(\Phi) = \textrm{Id}(2). $ For any other $j,k$ not already calculated,
$ F_{j,k} =0.$

Hence $\det(\nabla \varphi)(\Phi)$ is the product of the determinants of all $F_{k,k}$ for $k=1,\dots,n+1$, proving the claim.

Now define a second map, $\tilde{\varphi}: B_h(\Psi) \to (U \times \mathbb{R})^{n+1}$,
\[ \tilde{\varphi}(\Phi) : = ((x_0,v_0),(x_1,v_1),\dots,(x_n,v_n)). \]
This maps $\Phi$ to the initial position of each particle in $\Phi$. By repeatedly applying \eqref{eq-detgradphi},
\begin{equation} \label{eq-dettildphi}
\det (\nabla \tilde{\varphi}) (\Phi) = \prod_{j=1}^{n(\Phi)} \big(\varepsilon^2 (v(t_j^-) - v_j)\cdot \nu_j \big).
\end{equation}
For $h>0$ and $j=0,\dots,n$ define $C_{h,j}(\Phi) = C_{h,j} \subset U\times \mathbb{R}^3$ to be the cube with side length $h$ centred at $\tilde{\varphi}_j(\Phi)$.
Further for $h>0$ define,
\[ C_h(\Phi) := \prod_{j=0}^{n} C_{h,j}(\Phi).   \]
By the fact that the probability of finding a tree is less than the probability that initially there is a particles at the required initial position,
\begin{align} \label{eq-phatchbound}
\hat{P}_t(\tilde{\varphi}^{-1}(C_h))
&  \leq \int_{C_{h,0}}f_0(x,v) \, \mathrm{d}x \, \mathrm{d}v \nonumber \\
& \qquad \qquad  \times \prod_{j=1}^n  N \int_{C_{h,j}}g_0(v) \, \mathrm{d}x \, \mathrm{d}v.
\end{align}
Recalling that $\lambda$ denotes the Lebesgue measure on $\mathcal{MT}$, by \eqref{eq-dettildphi} it follows,
\begin{equation} \label{eq-lambdabound}
\lambda (\tilde{\varphi}^{-1}(C_h)) = \frac{h^{6(n+1)}}{\prod_{j=1}^{n} \big(\varepsilon^2 (v(t_j^-) - v_j)\cdot \nu_j \big)}(1+o(1)).
\end{equation}
Hence combining \eqref{eq-phatchbound} and \eqref{eq-lambdabound} and recalling \eqref{eq-boltzgrad},
\begin{align}
\frac{\hat{P}_t(\tilde{\varphi}^{-1}(C_h))}{\lambda (\tilde{\varphi}^{-1}(C_h))} & \leq \frac{1}{h^6} \int_{C_{h,0}}f_0(x,v) \, \mathrm{d}x \, \mathrm{d}v \nonumber \\
&  \qquad   \times   \prod_{j=1}^n \left( \frac{(v(t_j^-) - v_j)\cdot \nu_j}{h^6(1+o(1))}  \int_{C_{h,j}}g_0(v) \, \mathrm{d}x \, \mathrm{d}v  \right).
\end{align}
Since $f_0\in L^1(U\times \mathbb{R}^3)$ and $g_0\in L^1(\mathbb{R}^3)$ let $h$ tend to zero and the left hand side, which becomes $\hat{P}_t(\Phi) / \lambda (\Phi)$ in the limit, remains bounded. This completes the proof.
\end{proof}
We now prove the initial condition requirement on $\hat{P}_t$.
\begin{lem} \label{lem-emp1}
Under the assumptions and set up of theorem~\ref{thm-emp} we have $$\hat{P}_0(\Phi) = \zeta(\varepsilon)  f_0 (x_0,v_0) \mathbbm{1}_{n(\Phi) = 0}.$$
\end{lem}
\begin{proof}
In the case $n(\Phi) > 0$, we have $\hat{P}_0(\Phi)=0$. This is because the tree is free from initial overlap and the tree involves collisions happening at some positive time therefore the collisions cannot have occurred at time 0.

Now consider $n(\Phi)=0$. In this situation the tree $\Phi$ contains only the root particle and the probability of finding the root at the given initial data $(x_0,v_0)$ is given by $f_0(x_0,v_0)$. However this must be multiplied by a factor less than one because we must rule out situations that would give initial overlap of the root particle with a background particle. So we calculate the probability that there is no overlap. Firstly,
 \begin{align*}
\mathbb{P}(|x_0 - x_1| > \varepsilon ) & = 1 - \mathbb{P}(|x_0 - x_1| < \varepsilon )
 = 1 - \int_{\mathbb{R}^3} \int_{|x_0 - x_1| <\varepsilon}  g_0(\bar{v}) \, \mathrm{d}{x_1} \, \mathrm{d}\bar{v} \\
&= 1- \frac{4}{3}\pi \varepsilon^3 \int_{\mathbb{R}^3} g_0(\bar{v}) \, \mathrm{d}\bar{v}  = 1 - \frac{4}{3}\pi \varepsilon^3.
\end{align*}
Hence,
\begin{align*}
\mathbb{P}(|x_0 - x_j| > \varepsilon, \forall j = 1 ,\dots ,N ) & = \mathbb{P}(|x_0 - x_1| > \varepsilon ) ^N =( 1 - \frac{4}{3}\pi \varepsilon^3) ^N = \zeta (\varepsilon),
\end{align*} as required.
\end{proof}
Before we prove the loss term lemma we require a few technical estimates to calculate the rate at which the root particle experiences a collision.
\begin{deff} \label{def-wht}
Let $\Phi \in \mathcal{G}(\varepsilon)$. Define for $h>0$,
\begin{align*}
 W_h(t):= \Big\{ (\bar{x},\bar{v}) & \in U \times \mathbb{R}^3 : \exists (\nu',t') \in \mathbb{S}^2 \times (t,t+h) \\ & \textrm{ such that } x(t')+ \varepsilon \nu' = \bar{x }+t' \bar{v}  \textrm{ and } (v(t'^-) - \bar{v})\cdot \nu' >0 \Big\}.
\end{align*}
That is $W_h(t)$ is the set of initial points in $U \times \mathbb{R}^3 $ for the background particles that lead to a collision with the root particle of $\Phi$ between the time $t$ and $t+h$. Further define,
\begin{equation} \label{eq-iht}
I_h(t):= \frac{\int_{U \times \mathbb{R}^{3}}    g_0(\bar{v}) \mathbbm{1}_{W_h(t)} (\bar{x},\bar{v}) \mathbbm{1}_t^\varepsilon[\Phi] (\bar{x},\bar{v}) \, \mathrm{d}\bar{x} \, \mathrm{d}\bar{v} }{\int_{ U \times \mathbb{R}^3}  g_0(\bar{v}) \mathbbm{1}_t^\varepsilon[\Phi] (\bar{x},\bar{v}) \, \mathrm{d}\bar{x} \, \mathrm{d}\bar{v}  }.
\end{equation}
\end{deff}
From now on assume that the functions $V$ and $M$ in definition~\ref{def-goodtrees} satisfy, for any $0<\varepsilon<1$,
\begin{equation}\label{eq-vassump2}
\varepsilon V(\varepsilon)^3  \leq \frac{1}{8},  \end{equation}
and,
\begin{equation} \label{eq-massump}
M(\varepsilon)  \leq \frac{1}{\sqrt[]{\varepsilon}}.
\end{equation}
\begin{lem} \label{lem-phwht2coll}
Recall definition~\ref{deff-omega}. For $\varepsilon$ sufficiently small, $\Phi \in \mathcal{G}(\varepsilon)$ and $t>\tau$,
\begin{equation*}
\lim_{h \to 0}\frac{1}{h} \hat{P}_{t}(\#(\omega \cap W_h(t))\geq 2 \, | \, \Phi )   =0 . \end{equation*}
\end{lem}
\begin{proof}
Note that by the inclusion exclusion principle the fact that the background particles are independent,
\begin{align} \label{eq-phatwht2est}
 \hat{P}_{t}(&\#(\omega \cap  W_h(t))\geq 2 \, | \, \Phi ) \leq \sum_{1\leq i < j \leq N-n(\Phi)} \hat{P}_t( (x_i,v_i)\in W_h(t) \textrm{ and } (x_j,v_j) \in W_h(t) \, | \, \Phi) \nonumber \\
 & \leq N(N-1) \hat{P}_t( (x_1,v_1)\in W_h(t) \textrm{ and } (x_2,v_2) \in W_h(t) \, | \, \Phi) \nonumber \\
 &= N(N-1) \hat{P}_t( (x_1,v_1)\in W_h(t) \, | \, \Phi ) ^2.
\end{align}
Recalling \eqref{eq-iht},
\begin{equation} \label{eq-phatwhest}
\hat{P}_t( (x_1,v_1)\in W_h(t) \, | \, \Phi )= I_h(t).
\end{equation}
Now we estimate the right hand side of \eqref{eq-phatwhest} by estimating the numerator and denominator. Firstly by calculating the volume of the appropriate cylinder, for any $\bar{v} \in \mathbb{R}^3$,
\begin{equation} \label{eq-whtest}
\int_{U} \mathbbm{1}_{W_h(t)} (\bar{x},\bar{v}) \, \mathrm{d}\bar{x} \leq \pi \varepsilon^2 \int_t^{t+h} |v(s)-\bar{v}| \, \mathrm{d}s.
\end{equation}
Define
\begin{equation} \label{eq-beta}
\beta := \int_{\mathbb{R}^3} g_0(v)(1+|v|)\,\mathrm{d}v.
\end{equation}
Note that by assumption \eqref{eq-g0l1assmp}, $\beta<\infty$. Since $\Phi \in \mathcal{G}(\varepsilon)$ it follows that $|v(t)| \leq \mathcal{V}(\Phi) \leq V(\varepsilon)$. Using these and \eqref{eq-whtest} we estimate the numerator in \eqref{eq-phatwhest},

\begin{align} \label{eq-whtnumest}
\int_{U \times \mathbb{R}^{3}}    g_0(\bar{v}) \mathbbm{1}_{W_h(t)}& (\bar{x},\bar{v}) \mathbbm{1}_t^\varepsilon[\Phi] (\bar{x},\bar{v}) \, \mathrm{d}\bar{x} \, \mathrm{d}\bar{v}  \leq \int_{U \times \mathbb{R}^{3}}    g_0(\bar{v}) \mathbbm{1}_{W_h(t)} (\bar{x},\bar{v})  \, \mathrm{d}\bar{x} \, \mathrm{d}\bar{v} \nonumber \\
& \leq \int_{\mathbb{R}^3} g_0(\bar{v}) \pi \varepsilon^2 \int_t^{t+h} |v(s)-\bar{v}| \, \mathrm{d}s  \, \mathrm{d}\bar{v}
 \leq \pi \varepsilon^2 \int_{\mathbb{R}^3} g_0(\bar{v})  \int_t^{t+h} |v(s)| + | \bar{v}| \, \mathrm{d}s  \, \mathrm{d}\bar{v} \nonumber \\
& \leq  \pi \varepsilon^2 \int_{\mathbb{R}^3} g_0(\bar{v})  \int_t^{t+h} V(\varepsilon) + | \bar{v}| \, \mathrm{d}s  \, \mathrm{d}\bar{v}
 \leq  \pi \varepsilon^2 \int_{\mathbb{R}^3} g_0(\bar{v})  h\left(V(\varepsilon) + | \bar{v}| \right) \, \mathrm{d}\bar{v} \nonumber \\
& \leq  h \pi \varepsilon^2 \int_{\mathbb{R}^3} g_0(\bar{v})  \left( V(\varepsilon) + | \bar{v}| \right) \, \mathrm{d}\bar{v}
 \leq h \pi \varepsilon^2 (V(\varepsilon)+\beta).
\end{align}
Turning to the denominator of \eqref{eq-iht}. Firstly note that,
\begin{align} \label{eq-whtdenom1}
\int_{ U \times \mathbb{R}^3}  g_0(\bar{v}) \mathbbm{1}_t^\varepsilon[\Phi] (\bar{x},\bar{v}) \, \mathrm{d}\bar{x} \, \mathrm{d}\bar{v} &  = \int_{ U \times \mathbb{R}^3}  g_0(\bar{v}) \left ( 1 - \mathbbm{1}_{W_t(0)} (\bar{x},\bar{v}) \right) \, \mathrm{d}\bar{x} \, \mathrm{d}\bar{v} \nonumber \\
& = 1-  \int_{ U \times \mathbb{R}^3}  g_0(\bar{v})   \mathbbm{1}_{W_t(0)} (\bar{x},\bar{v})  \, \mathrm{d}\bar{x} \, \mathrm{d}\bar{v}.
\end{align}
By using \eqref{eq-whtest}, $t\leq T$ and the same estimates from the numerator estimate,
\begin{align} \label{eq-whtdenom2}
\int_{ U \times {R}^3}  g_0(\bar{v})   \mathbbm{1}_{W_t(0)} & (\bar{x},\bar{v})    \, \mathrm{d}\bar{x} \, \mathrm{d}\bar{v} \nonumber \\
& \leq \pi \varepsilon^2 \int_{\mathbb{R}^3}  g_0(\bar{v})   \int_0^{t} |v(s)-\bar{v}| \, \mathrm{d}s  \, \mathrm{d}\bar{v} \nonumber
 \leq  \pi \varepsilon^2 \int_{\mathbb{R}^3}  g_0(\bar{v})   \int_0^{t} \mathcal{V}(\Phi)+|\bar{v}| \, \mathrm{d}s  \, \mathrm{d}\bar{v} \nonumber \\
&\leq  \pi \varepsilon^2 \int_{\mathbb{R}^3}  g_0(\bar{v})   t \left( V(\varepsilon)+|\bar{v}| \right) \, \mathrm{d}\bar{v}
 \leq  \pi \varepsilon^2 T (V(\varepsilon)+\beta).
\end{align}
Hence for $\varepsilon$ sufficiently small by \eqref{eq-vassump2},
\[ \pi \varepsilon^2 T (V(\varepsilon)+\beta) \leq 1/2, \]
so by \eqref{eq-whtdenom1} and \eqref{eq-whtdenom2},
\begin{equation} \label{eq-whtdenomest}
\int_{ U \times \mathbb{R}^3}  g_0(\bar{v}) \mathbbm{1}_t^\varepsilon[\Phi] (\bar{x},\bar{v}) \, \mathrm{d}\bar{x} \, \mathrm{d}\bar{v}  \geq 1/2.
\end{equation}
Bounds for both the numerator and the denominator of \eqref{eq-phatwhest} have been found in equations \eqref{eq-whtnumest} and \eqref{eq-whtdenomest} respectively. Hence,
\begin{equation} \label{eq-ihtbound}
I_h(t) \leq 2  h \pi \varepsilon^2 (V(\varepsilon)+\beta).
\end{equation}
Substituting this into \eqref{eq-phatwht2est} and recalling \eqref{eq-boltzgrad},
\begin{align*}
\hat{P}_{t}(\#(\omega \cap W_h(t))\geq 2 \, | \, \Phi ) & \leq N(N-1 ) \times 4  h^2 \pi^2 \varepsilon^4 (V(\varepsilon)+\beta)^2 \\
& \leq 4   h^2 \pi^2 N^2  \varepsilon^4 (V(\varepsilon)+\beta)^2 \leq 4 h^2 \pi^2 (V(\varepsilon)+\beta)^2.
\end{align*}
This gives finally that,
\begin{align*}
 \lim_{h \to 0} \frac{1}{h} \left( \hat{P}_{t}(\#(\omega \cap W_h(t))\geq 2 \, | \, \Phi ) \right) & \leq  \lim_{h \to 0} \frac{1}{h}  \times 4 h^2 \pi^2 (V(\varepsilon)+\beta)^2 \\
 & \leq 4  \pi^2 (V(\varepsilon)+\beta)^2  \lim_{h \to 0}  h= 0,
\end{align*}
completing the proof of the lemma.
\end{proof}
The previous lemma shows that the rate of seeing two collisions in a short time converges to zero. We now show that the rate of seeing one collision converges to the required loss term. Before we do this we first estimate the error caused by re-collisions.
\begin{deff} \label{def-bht}
For $\Phi \in \mathcal{G}(\varepsilon)$, $t>\tau$ and $h>0$ recall the definition of $W_h(t)$ in definition~\ref{def-wht} and $\mathbbm{1}_t^\varepsilon[\Phi] (\bar{x},\bar{v})$ \eqref{eq-deff1phi}. Define, $B_{h,t}(\Phi) \subset U \times \mathbb{R}^3$,
\begin{align*}
B_{h,t}(\Phi) := \big\{ (\bar{x},\bar{v})\in U \times \mathbb{R}^3 :    \mathbbm{1}_t^\varepsilon[\Phi] (\bar{x},\bar{v}) = 0 \textrm{ and } \mathbbm{1}_{W_h(t)} (\bar{x},\bar{v}) =1 \big\}.
\end{align*}
Notice that $B_{h,t}(\Phi)$ is the set of all initial positions that a background particle can take such that it collides with the root once during $(0,t)$ and once during $(t,t+h)$.
\end{deff}
\begin{lem} \label{lem-intg0bhtbound}
For $\varepsilon$ sufficiently small, $\Phi \in \mathcal{G}(\varepsilon)$, $t>\tau$ and $h>0$ sufficiently small there exists a $\hat{C}_2(\varepsilon)>0$ depending on $t$ and $\Phi$ with $\hat{C}_2(\varepsilon)= o(1)$ as $\varepsilon$ tends to zero such that,
\[ \int_{B_{h,t}(\Phi)} g_0(\bar{v}) \, \mathrm{d}\bar{x} \, \mathrm{d}\bar{v} = h \varepsilon^2\hat{C}_2(\varepsilon).   \]
\end{lem}
\begin{proof}
Recall that $(x(t),v(t))$ is the trajectory of the tagged particle defined by the tree $\Phi$. $B_{h,t}(\Phi)$ is given by,
\begin{align*}
B_{h,t}(\Phi) = \big\{ (\bar{x}, \bar{v}) : & \,   \exists s\in (0,t), \sigma \in (t,t+h), \nu_1,\nu_2 \in \mathbb{S}^2  \textrm{ such that } \\
&  \bar{x} + s \bar{v} - x(s) =  \varepsilon \nu_1, \, \,  \bar{x} + \sigma \bar{v} - x(\sigma) =\varepsilon \nu_2    \textrm{ and } \\
& (v(s)-\bar{v})\cdot \nu_1 >0, \, \, (v(\sigma)-\bar{v})\cdot \nu_2 >0 \big\}.
\end{align*}
Define $\delta := \varepsilon^{1/3}$. We split the set $B_{h,t}(\Phi)$ into two parts, the first, denoted $B_{h,t}^\delta(\Phi)$, which considers $s\in(0,t-\delta]$ and the second, denoted $B_{h,t}^2(\Phi)$, which considers $s \in (t-\delta,t)$. We evaluate the bounds on these two sets separately.

Consider $s\in (0,t-\delta]$ and $ \sigma\in(t,t+h)$ fixed. The conditions defined by $B_{h,t}$ require that,
\[ \bar{v} = \frac{x(\sigma)- x(s)}{\sigma - s} + \frac{\varepsilon\nu_2 - \varepsilon \nu_1}{\sigma - s}. \]
For $\sigma$ fixed this implies that $\bar{v}$ is contained in a cylinder of radius $2\varepsilon/\delta$ around the curve defined by $\frac{x(\sigma)- x(s)}{\sigma - s}$ for $s \in (0,t-\delta]$. Recalling the definition of $V(\varepsilon)$ from definition~\ref{def-goodtrees}, taking $h \ll \varepsilon/ V(\varepsilon)$ implies $h|v(\sigma)| \ll \varepsilon$ so the dependence on $\sigma \in (t,t+h)$ gives only a small perturbation around the curve defined by $\frac{x(t)- x(s)}{t - s}$ for $s \in (0,t-\delta]$. Hence $\bar{v}$ is contained in the cylinder with radius $4\varepsilon/\delta$ around the piecewise differentiable curve $r(s):=\frac{x(t)- x(s)}{t - s}$ for $s \in (0,t-\delta]$.

Denote this cylinder in $\mathbb{R}^3$ by $E=E(t,\Phi,\delta)$. We seek a bound on the volume of $E$, $|E|$. First consider the length of the curve $r$. For almost all $s \in (0,t-\delta)$,
\[  \frac{\mathrm{d}}{\mathrm{d}s}r(s )= \frac{x(t)-x(s)}{(t-s)^2} + \frac{v(s)}{t-s} .\]
Hence,
\begin{align*}
|\frac{\mathrm{d}}{\mathrm{d}s}r(s )| & \leq \frac{|x(t)-x(s)|}{(t-s)^2}+ \frac{|v(s)|}{t-s}  \leq \frac{3}{(t-s)^2} + \frac{V(\varepsilon)}{t-s}.
\end{align*}
Thus the length of the curve is bounded by,
\begin{align*}
\int_0^{t-\delta} |\frac{\mathrm{d}}{\mathrm{d}s}r(s )| \, \mathrm{d}s &  \leq \int_0^{t-\delta} \frac{3}{(t-s)^2} + \frac{V(\varepsilon)}{t-s} \, \mathrm{d}s
 = \frac{3}{\delta} - \frac{3}{t} - V(\varepsilon)(\log(\delta) - \log(t)).
\end{align*}
Therefore for some $C>0$,
\begin{equation} \label{eq-volumeE}
|E| \leq  C \left(\frac{\varepsilon}{\delta}\right)^2 \left(\frac{3}{\delta} - \frac{3}{t} - V(\varepsilon)(\log(\delta) - \log(t)) \right).
\end{equation}
Noting that $x(\sigma) = x(t)+(\sigma-t) v(t)$, for $\bar{v}$ given, $(\bar{x},\bar{v}) \in B_{h,t} (\Phi)$ requires that,
\begin{align*}
\bar{x} & = x(\sigma) -\sigma\bar{v} + \varepsilon\nu_2 = x(t) + (\sigma -t) - \sigma \bar{v} + \varepsilon\nu_2 \\
& = x(t)-t\bar{v} +(\sigma- t)(v(t)-\bar{v})+\varepsilon\nu_2.
\end{align*}
Hence for $\bar{v}$ given $\bar{x}$ is contained in cylinder of radius $\varepsilon$ and length $h|v(t)-\bar{v}|$. Denote this cylinder by $C(\bar{v})$. By \eqref{eq-g0linf5} or \eqref{eq-g0l4eta}, for constants $C$ that change on each line,
\begin{align*}
\int_{B_{h,t}^\delta(\Phi) }g(\bar{v}) \, \mathrm{d}\bar{x} \, \mathrm{d}\bar{v} & \leq \int_E g(\bar{v}) \int_{C(\bar{v})} \, \mathrm{d}\bar{x} \, \mathrm{d}\bar{v} \leq C \varepsilon^2 h \int_E g(\bar{v}) |v(t)-\bar{v}| \, \mathrm{d}\bar{v} \\
&  \leq C \varepsilon^2 h \int_E g(\bar{v}) (V(\varepsilon)+ |\bar{v}|) \, \mathrm{d}\bar{v}  \leq C \varepsilon^2 h (V(\varepsilon)+1 ) |E|.
 \end{align*}
It remains to show that $(V(\varepsilon)+1)|E|$ is $o(1)$ as $\varepsilon$ tends to zero. Recall \eqref{eq-vassump2}, that $\delta = \varepsilon^{1/3}$ and \eqref{eq-volumeE},
 \begin{align*}
(V(\varepsilon)+1)|E|
& = (V(\varepsilon)+1) \times C  \left(\frac{\varepsilon}{\delta}\right)^2 \left(\frac{3}{\delta} - \frac{3}{t} - V(\varepsilon)(\log(\delta) - \log(t)) \right) \\
& \leq C\left(\frac{1}{2\varepsilon^{1/3}}+1\right)\varepsilon^{4/3} \left(\frac{3}{\varepsilon^{1/3}} +\frac{1}{2\varepsilon^{1/3}}(|\log\varepsilon^{1/3}|+|\log t|)\right) \\
& \leq C \left(\frac{1}{2\varepsilon^{1/3}}+1\right) \left(3\varepsilon +\frac{1}{2}\varepsilon\left(\frac{1}{3}|\log\varepsilon|+|\log t|\right)\right)
\end{align*}
as required.

Now consider the second part of $B_{h,t}(\Phi)$ for $s \in (t-\delta,t)$ denoted $B_{h,t}^2(\Phi)$. Since $\Phi$ is fixed, $t > \tau$ and $\delta = \varepsilon^{1/3}$ let $\varepsilon$ sufficiently small such that $t-\delta >\tau$. Hence for $s \in (t-\delta,t)$, $v(s)=v(t)$. We change the velocity space coordinates so that $v(t)=0$. If we require that a particle starting at $(\bar{x},\bar{v})$ collides with the tagged particle in $(t-\delta,t)$ and again in $(t,t+h)$ we require in the new coordinates that either $\bar{v}=0$ or that $|\bar{v}|$ is sufficiently large so that the background particle wraps round the torus having travelled at least distance $3/4$ (for $\varepsilon$ sufficiently small) within time $(\delta +h)$ to re-collide with the tagged particle. That is,
\[ |\bar{v}| \geq \frac{3}{4(\delta + h)} . \]
For $h \leq 1/4  \delta $, this implies $|\bar{v}| \geq \frac{3}{5\delta}$.
Changing back to the original coordinates, this means it is required that $\bar{v}=v(t)$ or $|\bar{v}-v(t)| \geq 3/5\delta$.

For a given $\bar{v}$, the same conditions as before on the $\bar{x}$ coordinate must hold and so $\bar{x}$ is in the cylinder $C(\bar{v})$. Recalling \eqref{eq-g0l1assmp}, \eqref{eq-vassump2},  $\delta=\varepsilon^{1/3}$ and that $|v(t)| \leq V(\varepsilon) \leq 1/2\varepsilon^{-1/3} = 1/2\delta$ it follows for constants $C$ that change on each line,
\begin{align*}
\int_{B_{h,t}^2(\Phi)} g_0(\bar{v}) \, &\mathrm{d}\bar{x} \, \mathrm{d}\bar{v}  \leq \int_{\mathbb{R}^3 \setminus B_{3/5\delta}(v(t))} g_0(\bar{v}) \int_{C(\bar{v})} \, \mathrm{d}\bar{x} \, \mathrm{d}\bar{v}
 \leq C\varepsilon^2h \int_{\mathbb{R}^3 \setminus B_{3/5\delta}(v(t))} g_0(\bar{v})|v(t)-\bar{v}| \, \mathrm{d}\bar{v} \\
& \leq C\varepsilon^2h \int_{\mathbb{R}^3 \setminus B_{3/5\delta}(v(t))} g_0(\bar{v})(V(\varepsilon)+|\bar{v}|) \, \mathrm{d}\bar{v}
 \leq C\varepsilon^2h \int_{\mathbb{R}^3 \setminus B_{1/10\delta}(0)} g_0(\bar{v})(V(\varepsilon)+|\bar{v}|) \, \mathrm{d}\bar{v} \\
& \leq C\varepsilon^2h  \int_{\mathbb{R}^3 \setminus B_{1/10\delta}(0)} g_0(\bar{v})(100\delta^2 |\bar{v}|^2 V(\varepsilon)+10\delta|\bar{v}|^2)  \, \mathrm{d}\bar{v}  \leq C\varepsilon^2h  (10\delta^2 V(\varepsilon)+\delta )\\
& \leq C\varepsilon^2 h \left(5 \frac{\varepsilon^{2/3}}{\varepsilon^{1/3}}+\varepsilon^{1/3}  \right)  =  C\varepsilon^2 h \times  \varepsilon^{1/3},
\end{align*}
as required. Since $B_{h,t}(\Phi) = B_{h,t}^\delta(\Phi) \cup B_{h,t}^2(\Phi)$ the proof of the lemma is complete.

\end{proof}

\begin{lem} \label{lem-phwht1coll}
For $\varepsilon$ sufficiently small, $\Phi \in \mathcal{G}(\varepsilon)$, $t>\tau$ and $\hat{C}_2(\varepsilon)$ as in the above lemma,
\begin{align*}
\lim_{h \to 0}\frac{1}{h} \hat{P}_{t} & (\#(\omega \cap W_h(t)) = 1 \, | \, \Phi )   \\
& =(1- \gamma(t)) \frac{\int_{\mathbb{S}^{2}}  \int_{\mathbb{R}^3}   g_0(\bar{v})[(v(t)-\bar{v})\cdot \nu]_+ \, \mathrm{d}\bar{v} \, \mathrm{d}\nu - \hat{C}_2(\varepsilon) }{\int_{ U \times {R}^3}  g_0(\bar{v}) \mathbbm{1}_t^\varepsilon[\Phi] (\bar{x},\bar{v}) \, \mathrm{d}\bar{x} \, \mathrm{d}\bar{v}  }.
\end{align*}
\end{lem}
\begin{proof}
Since the initial data for each background particle is  independent of the other background particles,
\begin{align} \label{eq-omegawht1}
\hat{P}_{t}(\#(\omega \cap & W_h(t)) = 1 \, | \, \Phi ) =\sum_{i=1}^{N-n(\Phi)} \hat{P}_t \Big( (x_i,v_i) \in W_h(t) \textrm{ and } (x_1,v_1), \dots,  \nonumber \\
& \qquad \qquad (x_{i-1},v_{i-1}),(x_{i+1},v_{i+1}), \dots,  (x_{N-n(\Phi)},v_{N-n(\Phi)}) \notin W_h(t) \, | \, \Phi  \Big) \nonumber \\
&
= (N-n(\Phi))\hat{P}_t((x_1,v_1) \in W_h(t) \, | \, \Phi)\hat{P}_t((x_2,v_2) \notin W_h(t) \, | \, \Phi)^{N-n(\Phi) -1}\nonumber \\
& = (N-n(\Phi))I_h(t) \left( 1-I_h(t) \right)^{N-n(\Phi) -1} \nonumber \\& = (N-n(\Phi))I_h(t) \sum_{j=0}^{N-n(\Phi)-1} (-1)^{j} \binom{N-n(\Phi)-1}{j} I_h(t)^j \nonumber \\
& = (N-n(\Phi))\sum_{j=0}^{N-n(\Phi)-1} (-1)^{j} \binom{N-n(\Phi)-1}{j} I_h(t)^{j+1}.
\end{align}
By \eqref{eq-ihtbound},
\[ \lim_{h \to 0} \frac{1}{h}I_h(t)^2 = 0. \]
Hence dividing \eqref{eq-omegawht1} by $h$ and taking $h$ to zero we see that all terms in the sum for $j\geq 1$ tend to zero, leaving only the contribution from the term $j=0$. Hence,
\begin{align} \label{eq-wht1calc}
\lim_{h \to 0} \frac{1}{h} \hat{P}_{t} & (\#(\omega \cap  W_h(t)) = 1 \, | \, \Phi ) \nonumber \\
& = \lim_{h \to 0} \frac{1}{h} (N-n(\Phi))\sum_{j=0}^{N-n(\Phi)-1} (-1)^{j} \binom{N-n(\Phi)-1}{j} I_h(t)^{j+1}
= \lim_{h \to 0} \frac{1}{h} (N-n(\Phi)) I_h(t)  .
\end{align}
It remains to investigate,
\[ \lim_{h \to 0} \frac{1}{h}  I_h(t) = \lim_{h \to 0} \frac{1}{h} \frac{\int_{U \times \mathbb{R}^{3}}    g_0(\bar{v}) \mathbbm{1}_{W_h(t)} (\bar{x},\bar{v}) \mathbbm{1}_t^\varepsilon[\Phi] (\bar{x},\bar{v}) \, \mathrm{d}\bar{x} \, \mathrm{d}\bar{v} }{\int_{ U \times \mathbb{R}^3}  g_0(\bar{v}) \mathbbm{1}_t^\varepsilon[\Phi] (\bar{x},\bar{v}) \, \mathrm{d}\bar{x} \, \mathrm{d}\bar{v}  }.    \]
For $B_{h,t}(\Phi)$ as defined in definition~\ref{def-bht} we have,
\begin{align*}
\int_{U \times \mathbb{R}^{3}}    g_0(\bar{v}) &  \mathbbm{1}_{W_h(t)} (\bar{x},\bar{v}) \mathbbm{1}_t^\varepsilon[\Phi] (\bar{x},\bar{v}) \, \mathrm{d}\bar{x} \, \mathrm{d}\bar{v}  \\
& = \int_{U \times \mathbb{R}^{3}}    g_0(\bar{v})  \mathbbm{1}_{W_h(t)} (\bar{x},\bar{v})  \, \mathrm{d}\bar{v}  - \int_{B_{h,t}(\Phi)}    g_0(\bar{v})  \, \mathrm{d}\bar{x} \, \mathrm{d}\bar{v}
\end{align*}
It then follows from lemma~\ref{lem-intg0bhtbound},
\begin{align*}
\lim_{h \to 0} \frac{1}{h}  I_h(t) & = \lim_{h \to 0} \frac{1}{h} \frac{\int_{U \times \mathbb{R}^{3}}    g_0(\bar{v}) \mathbbm{1}_{W_h(t)} (\bar{x},\bar{v}) \mathbbm{1}_t^\varepsilon[\Phi] (\bar{x},\bar{v}) \, \mathrm{d}\bar{x} \, \mathrm{d}\bar{v} }{\int_{ U \times \mathbb{R}^3}  g_0(\bar{v}) \mathbbm{1}_t^\varepsilon[\Phi] (\bar{x},\bar{v}) \, \mathrm{d}\bar{x} \, \mathrm{d}\bar{v}  } \\
& = \lim_{h \to 0} \frac{1}{h} \left( \frac{h  \varepsilon^2 \int_{\mathbb{S}^2}  \int_{\mathbb{R}^{3}}    g_0(\bar{v})  [(v(t)-\bar{v})\cdot \nu ]_+ \,   \mathrm{d}\bar{v} \, \mathrm{d}\nu -h\varepsilon^2\hat{C}_2(\varepsilon)}{\int_{ U \times \mathbb{R}^3}  g_0(\bar{v}) \mathbbm{1}_t^\varepsilon[\Phi] (\bar{x},\bar{v}) \, \mathrm{d}\bar{x} \, \mathrm{d}\bar{v}} \right) \\
& = \varepsilon^2 \frac{   \int_{\mathbb{S}^2}  \int_{\mathbb{R}^{3}}    g_0(\bar{v})  [(v(t)-\bar{v})\cdot \nu ]_+ \,   \mathrm{d}\bar{v} \, \mathrm{d}\nu - \hat{C}_2(\varepsilon) }{\int_{ U \times \mathbb{R}^3}  g_0(\bar{v}) \mathbbm{1}_t^\varepsilon[\Phi] (\bar{x},\bar{v}) \, \mathrm{d}\bar{x} \, \mathrm{d}\bar{v}}.
\end{align*}
Substituting this into \eqref{eq-wht1calc},
\begin{align*}
\lim_{h \to 0} \frac{1}{h} \hat{P}_{t} & (\#(\omega \cap  W_h(t)) = 1 \, | \, \Phi )  = \lim_{h \to 0} \frac{1}{h} (N-n(\Phi)) I_h(t)  \\
& = (N-n(\Phi))\varepsilon^2 \frac{   \int_{\mathbb{S}^2}  \int_{\mathbb{R}^{3}}    g_0(\bar{v})  [(v(t)-\bar{v})\cdot \nu ]_+ \,   \mathrm{d}\bar{v} \, \mathrm{d}\nu - \hat{C}_2(\varepsilon) }{\int_{ U \times \mathbb{R}^3}  g_0(\bar{v}) \mathbbm{1}_t^\varepsilon[\Phi] (\bar{x},\bar{v}) \, \mathrm{d}\bar{x} \, \mathrm{d}\bar{v}} \\
& = (1-\gamma(t))\frac{   \int_{\mathbb{S}^2}  \int_{\mathbb{R}^{3}}    g_0(\bar{v})  [(v(t)-\bar{v})\cdot \nu ]_+ \,   \mathrm{d}\bar{v} \, \mathrm{d}\nu- \hat{C}_2(\varepsilon) }{\int_{ U \times \mathbb{R}^3}  g_0(\bar{v}) \mathbbm{1}_t^\varepsilon[\Phi] (\bar{x},\bar{v}) \, \mathrm{d}\bar{x} \, \mathrm{d}\bar{v}},
\end{align*}
which proves the lemma.
\end{proof}
With these estimates we can now prove that the loss term of \eqref{eq-emp} holds.
\begin{lem} \label{lem-emp2}
Under the assumptions and set up of theorem~\ref{thm-emp}, for $t > \tau$,
\[ \partial_t \hat{P}_t(\Phi) = (1-\gamma(t))\hat{\mathcal{Q}}_t^- [\hat{P}_t] (\Phi). \]
\end{lem}
\begin{proof}
We calculate, for $t>\tau$,
\[ \partial_t \hat{P}_t (\Phi) = \lim_{h \to 0} \frac{1}{h} \left(\hat{P}_{t+h}(\Phi) -  \hat{P}_t(\Phi) \right). \]
Noting that,
\[ \hat{P}_{t+h}(\Phi) = \left(1-\hat{P}_{t}\left(\#(\omega \cap W_h(t))>0 \, | \, \Phi \right) \right) \hat{P}_t(\Phi). \]
It follows,
\begin{equation*}
\frac{1}{h}\left( \hat{P}_{t+h}(\Phi) -  \hat{P}_t(\Phi)\right) = -\frac{1}{h}\hat{P}_{t}(\#(\omega \cap W_h(t))>0 \, | \, \Phi )  \hat{P}_t(\Phi).
\end{equation*}
Using lemma~\ref{lem-phwht2coll} and lemma~\ref{lem-phwht1coll},
\begin{align*}
\partial_t \hat{P}_t (\Phi) & = \lim_{h \to 0} \frac{1}{h} \left(\hat{P}_{t+h}(\Phi) -  \hat{P}_t(\Phi)\right) \\
& = - \lim_{h \to 0} \frac{1}{h}\hat{P}_{t}(\#(\omega \cap W_h(t))>0 \, | \, \Phi )  \hat{P}_t(\Phi)
= -\hat{P}_t(\Phi) \lim_{h \to 0} \frac{1}{h}\hat{P}_{t}(\#(\omega \cap W_h(t))=1 \, | \, \Phi )    \\
&= -\hat{P}_t(\Phi)  (1- \gamma(t)) \frac{\int_{\mathbb{S}^{2}}  \int_{\mathbb{R}^3}   g_0(\bar{v})[(v(t)-\bar{v})\cdot \nu]_+ \, \mathrm{d}\bar{v} \, \mathrm{d}\nu - \hat{C}_2(\varepsilon) }{\int_{ U \times {R}^3}  g_0(\bar{v}) \mathbbm{1}_t^\varepsilon[\Phi] (\bar{x},\bar{v}) \, \mathrm{d}\bar{x} \, \mathrm{d}\bar{v}  }  = (1- \gamma(t))\hat{\mathcal{Q}}_t^- [\hat{P}_t] (\Phi),
\end{align*}
which proves the lemma.
\end{proof}
We next move to proving the gain term in \eqref{eq-emp}.
\begin{lem} \label{lem-emp3}
Under the assumptions and set up of theorem~\ref{thm-emp}, for $n(\Phi)\geq 1$
\[ \hat{P}_\tau (\Phi) =(1-\gamma(\tau))\hat{P}_\tau(\bar{\Phi})  \frac{g_0(v')[(v(\tau^-)-v')\cdot \nu ]_+}{1- \pi \varepsilon^2 \int_0^\tau |v(s)-v'| \, \mathrm{d}s + \hat{C}_1\varepsilon^3}.\]
For some $\hat{C}_1>0$ depending only on $\Phi$.
\end{lem}
\begin{proof}
Firstly,
\begin{align*}
\hat{P}_\tau(\Phi) & = \hat{P}_\tau(\Phi \cap \bar{\Phi})  = \hat{P}_\tau(\Phi \, |\,  \bar{\Phi}) \hat{P}_\tau( \bar{\Phi}) .
\end{align*}
It remains to show that
\begin{equation} \label{eq-phatgain}
\hat{P}_\tau(\Phi \, |\,  \bar{\Phi})=(1-\gamma(\tau)) \frac{g_0(v')[(v(\tau^-)-v')\cdot \nu ]_+}{1- \pi \varepsilon^2 \int_0^\tau |v(s)-v'| \, \mathrm{d}s + \hat{C}_1\varepsilon^3}.
\end{equation}
We do this by proving upper and lower bounds. For $h \geq 0$ define,
\[ U_h:= \{ \Psi \in \mathcal{MT} : \bar{\Psi}_{\tau(\Phi)} = \bar{\Phi}_{\tau(\Phi)} \textrm{ and } \Psi \in B_h(\Phi) \}. \]
Note that $U_0:= \{ \Phi \}$. Then by lemma~\ref{lem-phatabscts},
\begin{equation} \label{eq-phatlim}
\hat{P}_\tau (\Phi \, | \, \bar{\Phi}) = \lim_{h \to 0} h^{-6}\hat{P}_\tau (U_h \, | \, \bar{\Phi}).
\end{equation}
For $\Psi \in U_h$ define $V_h(\Psi) \in U \times \mathbb{R}^3$ to be the initial position of the background particle that leads to the final collision of $\Psi$ and define $V_h \subset U \times \mathbb{R}^3$ by,
\[ V_h = \cup_{\Psi \in U_h} V_h(\Psi). \]
Note that $V_0 = \{ (x(\tau)+\varepsilon \nu - \tau v',v') \}$ that is $V_0$ contains only the initial point of the background particle that gives the final collision in $\Phi$. Then by a change of coordinates, recalling \eqref{eq-detgradphi},
 \begin{align} \label{eq-phatuhlower}
\hat{P}_\tau(U_h \, | \, \bar{\Phi})  & \leq \sum_{i=1}^{N-(n(\Phi)-1)} \hat{P}_\tau ((x_i,v_i) \in V_h \, | \, \bar{\Phi} _\tau) \varepsilon^2 [(v(\tau^-) - v')\cdot \nu]_+ \nonumber \\
& = (N-(n(\Phi)-1))\varepsilon ^2\hat{P}_\tau ((x_1,v_1) \in V_h \, | \, \bar{\Phi} )   [(v(\tau^-) - v')\cdot \nu]_+ \nonumber \\
& =(1- \gamma(\tau) )\hat{P}_\tau ((x_1,v_1) \in V_h \, | \, \bar{\Phi} )   [(v(\tau^-) - v')\cdot \nu]_+.
\end{align}
By absolute continuity of $\hat{P}_\tau$,
\[ \lim_{h \to 0} h^{-6}\hat{P}_\tau ((x_1,v_1) \in V_h \, | \, \bar{\Phi}) = \hat{P}_\tau ((x_1,v_1) \in V_0 \, | \, \bar{\Phi}). \]
Combining these into \eqref{eq-phatlim},
\begin{align} \label{eq-phatgivenlower}
\hat{P}_\tau (\Phi \, | \, \bar{\Phi}) &= \lim_{h \to 0} h^{-6}\hat{P}_\tau (U_h \, | \, \bar{\Phi}) \nonumber \\
& \leq \lim_{h \to 0} h^{-6} (1- \gamma(\tau)) \hat{P}_\tau ((x_1,v_1) \in V_h \, | \, \bar{\Phi} )   [(v(\tau^-) - v')\cdot \nu]_+\nonumber \\
&
 = (1- \gamma(\tau)) \hat{P}_\tau ((x_1,v_1) \in V_0 \, | \, \bar{\Phi} _\tau)   [(v(\tau^-) - v')\cdot \nu]_+.
\end{align}
Next consider the lower bound. By the inclusion-exclusion principle,
\begin{align} \label{eq-phatuhupper}
\hat{P} _\tau (U_h \, | \, \bar{\Phi}) & \geq \sum_{i=1}^{N-(n(\Phi)-1)} \hat{P}_\tau ((x_i,v_i) \in V_h \, | \, \bar{\Phi}) \varepsilon^2 [(v(\tau^-) - v')\cdot \nu]_+  \nonumber \\
& \quad  - \sum_{1 \leq i < j \leq N-(n(\Phi)-1)} \hat{P}_\tau ((x_i,v_i), (x_j,v_j) \in V_h \, | \, \bar{\Phi}) \varepsilon^2 [(v(\tau^-) - v')\cdot \nu]_+ .
\end{align}
As in \eqref{eq-phatuhlower} it follows,
\begin{align*}
\sum_{i=1}^{N-(n(\Phi)-1)} \hat{P}_\tau & ((x_i,v_i) \in  V_h \, | \, \bar{\Phi}) \varepsilon^2 [(v(\tau^-) - v')\cdot \nu]_+ \\
& = (1-\gamma(\tau))\hat{P}_\tau ((x_1,v_1) \in V_h \, | \, \bar{\Phi}) [(v(\tau^-) - v')\cdot \nu]_+.
\end{align*}
Further,
\begin{align*}
\sum_{1 \leq i < j \leq N-(n(\Phi)-1)} \hat{P}_\tau & ((x_i,v_i), (x_j,v_j) \in V_h \, | \, \bar{\Phi}) \varepsilon^2 [(v(\tau^-) - v')\cdot \nu]_+ \nonumber
 \\
 & \leq N(N-1) \hat{P}_\tau ((x_1,v_1), (x_2,v_2) \in V_h \, | \, \bar{\Phi}) \varepsilon^2 [(v(\tau^-) - v')\cdot \nu]_+  \nonumber \\
 & = (N-1)\hat{P}_\tau ((x_1,v_1) \in V_h \, | \, \bar{\Phi})^2 [(v(\tau^-) - v')\cdot \nu]_+.
\end{align*}
By the absolute continuity of $\hat{P}_t$ this implies,
\begin{align*}
\lim_{h \to 0} \frac{1}{h^6} \sum_{1 \leq i < j \leq N-(n(\Phi)-1)} \hat{P}_\tau & ((x_i,v_i), (x_j,v_j) \in V_h \, | \, \bar{\Phi}) \varepsilon^2 [(v(\tau^-) - v')\cdot \nu]_+ =0.
\end{align*}
Hence by \eqref{eq-phatuhupper},
\begin{align}  \label{eq-phatgivenupper}
\hat{P}_\tau (\Phi \, | \, \bar{\Phi}) &= \lim_{h \to 0} h^{-6}\hat{P}_\tau (U_h \, | \, \bar{\Phi}) \nonumber \\
& \geq \lim_{h \to 0} h^{-6} (1-\gamma(\tau)) \hat{P}_\tau ((x_1,v_1) \in V_h \, | \, \bar{\Phi})  [(v(\tau^-) - v')\cdot \nu]_+ \nonumber \\
& = (1-\gamma(\tau)) \hat{P}_\tau ((x_1,v_1) \in V_0 \, | \, \bar{\Phi})  [(v(\tau^-) - v')\cdot \nu]_+.
\end{align}
Recalling that we need to prove \eqref{eq-phatgain} to prove the lemma, we see that with \eqref{eq-phatgivenlower} and \eqref{eq-phatgivenupper} we now need only to show that,
\begin{equation*}
\hat{P}_\tau ((x_1,v_1) \in V_0 \, | \, \bar{\Phi}) =  \frac{g_0(v')}{1- \pi \varepsilon^2 \int_0^\tau |v(s)-v'| \, \mathrm{d}s + \hat{C}_1\varepsilon^3}.
\end{equation*}
That is,
\begin{equation} \label{eq-phatx1v1}
\hat{P}_\tau \left( (x_1,v_1)=(x(\tau)+\varepsilon\nu - \tau v',v')\, | \, \bar{\Phi} \right) =  \frac{g_0(v')}{1- \pi \varepsilon^2 \int_0^\tau |v(s)-v'| \, \mathrm{d}s + \hat{C}_1\varepsilon^3}.
\end{equation}

Since we are conditioning on $\bar{\Phi}$ occurring there is a region of $U$ that we must rule out for the initial position of the background particle if we know that is has velocity $v'$. That is to say there is a region of $U$ where we know the background particle cannot have started with initial velocity $v'$ because if it did it would have interfered with $\bar{\Phi}$. Denote this region of $U$ by $\Delta$. To calculate the volume of $\Delta$ we imagine it as cylinders to get, for fixed incoming velocity $v'$,
\begin{equation} \label{eq-voldelta}
|\Delta| = \pi \varepsilon^2 \int_0^\tau |v(s)-v'| \, \mathrm{d}s - \hat{C}_1\varepsilon^3,
\end{equation}
where the first term calculates the volume of each cylinder as the root particle changes direction and the second term subtracts the over-estimate error caused counting certain parts twice as the particle changes direction.
\begin{figure}[ht]
\centering
\includegraphics[scale=1]{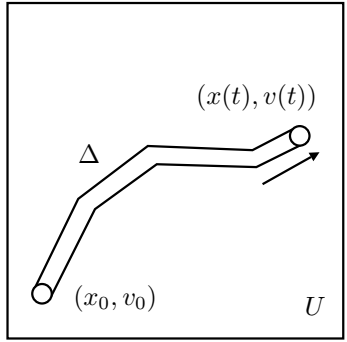}
\caption{In the case $v'=0$ we are calculating the volume of $\Delta$, since we know the background particle cannot start in $\Delta$. For $v'\neq 0$ the cylinders get shifted but the principle is the same. (Diagram not to scale)}
\label{fig-delta}
\end{figure}

Therefore, recalling that $|U|=1$,
\[ |U \setminus \Delta| = 1- \left( \pi \varepsilon^2 \int_0^\tau |v(s)-v'| \, \mathrm{d}s -\hat{C}_1\varepsilon^3\right), \]
and this together with the fact that the velocity of the background particle has initial distribution $g_0$ gives the required \eqref{eq-phatx1v1}.
\end{proof}

\begin{proof}[Proof of theorem~\ref{thm-emp}.]
Combining lemmas~\ref{lem-emp1}, \ref{lem-emp2} and \ref{lem-emp3} gives the required proof of the theorem.
\end{proof}

\section{Convergence}\label{sec:conv}
Having proven the existence of the idealized distribution $P_t$ and shown that the empirical distribution $\hat{P}_t$ solves the appropriate equation, we seek to show the convergence results that will help prove our main result. Following \cite{matt12}, the idea is to establish a differential inequality in \eqref{eq-phatrcompeq}. In combination with  the fact that $P_t$ is a probability measure and  that $\lim_{\varepsilon \to 0}P_t(\mathcal{G}(\varepsilon))=1$ in proposition \ref{prop-goodtreefull} the inequality delivers the convergence result theorem~\ref{thm-ptphatcomp}. The main theorem~\ref{thm-main} is a direct consequence. We first introduce some notation. Recall \eqref{eq-deff1phi} and \eqref{eq-zetadeff}. For $\varepsilon>0$, $\Phi \in \mathcal{G}(\varepsilon)$, $t \in [0,T]$, define the following,
\begin{align}
\eta_t^\varepsilon(\Phi) & := \int_{U \times \mathbb{R}^3} g_0(\bar{v})(1- \mathbbm{1}_t^\varepsilon[\Phi](\bar{x},\bar{v})) \, \mathrm{d}\bar{x} \, \mathrm{d}\bar{v}, \nonumber \\
 R_t^\varepsilon(\Phi) &: = \zeta(\varepsilon)P_t(\Phi), \nonumber \\
L(\Phi) &: = -\int_{\mathbb{S}^{2}} \int_{\mathbb{R}^3} g_0(\bar{v})[(v(\tau)-\bar{v})\cdot \nu ]_+ \, \mathrm{d}\bar{v} \, \mathrm{d}\nu, \nonumber \\
C(\Phi)& := 2 \sup_{t \in [0,T]} \left\{ \int_{\mathbb{S}^{2}} \int_{\mathbb{R}^3} g_0(\bar{v})[(v(t)-\bar{v})\cdot \nu ]_+ \right\} \label{eq-cdeff}  \\
 \rho^{\varepsilon,0}_t (\Phi) &: = \eta_t^\varepsilon(\Phi)C(\Phi)t . \nonumber
\end{align}
Further for $k \geq 1$ define,
\[  \rho^{\varepsilon,k}_t (\Phi) := (1-\varepsilon)\rho^{\varepsilon,k-1}_t (\Phi) + \rho^{\varepsilon,0}_t (\Phi) + \varepsilon.  \]
Note that this implies that for $k \geq 1$,
\begin{equation} \label{eq-rhoalt}
\rho^{\varepsilon,k}_t(\Phi) =(1-\varepsilon)^k \rho^{\varepsilon,0}_t (\Phi) + (\rho^{\varepsilon,0}_t (\Phi) + \varepsilon) \sum_{j=1}^k(1-\varepsilon)^{k-j}.
\end{equation}
Finally define,
\[ \hat{\rho}^\varepsilon_t(\Phi) := \rho^{\varepsilon,n(\Phi)}_t(\Phi). \]

\begin{prop} \label{prop-ptphatcomp}
For $\varepsilon$ sufficiently small, $\Phi \in \mathcal{G}(\varepsilon)$ and $t\in [0,T]$,
\begin{equation}\label{eq-phatrcompeq}
\hat{P}^\varepsilon_t(\Phi)-R^\varepsilon_t(\Phi) \geq - \hat{\rho}^\varepsilon_t(\Phi)R_t^\varepsilon(\Phi).
\end{equation}
\end{prop}
To prove this proposition we use a number of lemmas.
\begin{lem} \label{lem-pphatcomphelp}
For $\Phi \in \mathcal{G}(\varepsilon)$ and $t \geq \tau$,
\begin{align*}
\hat{P}_t^\varepsilon(\Phi) - R_t^\varepsilon(\Phi) & \geq \exp\left(L(\Phi)  \int_\tau^t (1+2\eta_s^\varepsilon(\Phi)) \, \mathrm{d}s \right)(\hat{P}_\tau^\varepsilon(\Phi) - R_\tau^\varepsilon(\Phi)) \\
&  \qquad + 2\eta_t^\varepsilon(\Phi) L(\Phi) R_t^\varepsilon(\Phi) \int_\tau^t \exp \left(2\eta_s^\varepsilon(\Phi) (t-s) L(\Phi) \right)   \, \mathrm{d}s.
\end{align*}
\end{lem}
\begin{proof}
For $t=\tau$ the result holds trivially. For $t >\tau $, by theorem~\ref{thm-ideq} and theorem~\ref{thm-emp},
\begin{equation} \label{eq-rtptcomp}
\partial_t \left( \hat{P}_t^\varepsilon(\Phi) -R_t^\varepsilon(\Phi)  \right) = (1-\gamma(t))\hat{L}_t(\Phi)\hat{P}_t(\Phi) -L(\Phi)R_t^\varepsilon(\Phi),
\end{equation}
where
\[ \hat{L}_t(\Phi) := - \left( \frac{\int_{\mathbb{S}^{2}}  \int_{\mathbb{R}^3}   g_0(\bar{v})[(v(\tau)-\bar{v})\cdot \nu]_+ \, \mathrm{d}\bar{v} \, \mathrm{d}\nu - \hat{C}_2(\varepsilon)}{\int_{ U \times \mathbb{R}^3}  g_0(\bar{v}) \mathbbm{1}_t^\varepsilon[\Phi] (\bar{x},\bar{v}) \, \mathrm{d}\bar{x} \, \mathrm{d}\bar{v}  } \right). \]
Recalling \eqref{eq-voldelta} and \eqref{eq-beta}, for $\varepsilon$ sufficiently small,
\begin{align} \label{eq-etabound}
\int_{U \times \mathbb{R}^3} g_0(\bar{v})(1 & - \mathbbm{1}_t^\varepsilon[\Phi](\bar{x},\bar{v})) \, \mathrm{d}\bar{x} \, \mathrm{d}\bar{v}
\leq \int_{\mathbb{R}^3} g_0(\bar{v})\left(\pi \varepsilon^2 \int_0^t |v(s)-\bar{v}| \, \mathrm{d}s - \hat{C}_1\varepsilon^3 \right) \,  \mathrm{d}\bar{v}
\nonumber \\ &\leq \int_{\mathbb{R}^3} g_0(\bar{v})\left(\pi \varepsilon^2 T(\mathcal{V}(\Phi)+ | \bar{v}|)  - \hat{C}_1\varepsilon^3 \right) \,  \mathrm{d}\bar{v} \leq \int_{\mathbb{R}^3} g_0(\bar{v})\left(\pi \varepsilon^2 T(V(\varepsilon)+ | \bar{v}|)  \right) \,  \mathrm{d}\bar{v} \nonumber \\
& \leq  \pi \varepsilon^2 T \int_{\mathbb{R}^3} g_0(\bar{v})\left(V(\varepsilon)+ | \bar{v}|  \right) \,  \mathrm{d}\bar{v} \leq \pi \varepsilon^2 T (V(\varepsilon)+\beta)  < 1/2.
\end{align}
Noting that for $0\leq z \leq 1/2$, $\sum_{i=0}^\infty z^i \leq 2$ which gives,
\begin{align*}
\frac{1}{1-z} & = \sum_{i=0}^\infty z^i = 1 + z\left( \sum_{i=0}^\infty z^i \right) \leq 1+ 2z.
\end{align*}
It follows,
\begin{align}
\frac{1}{\int_{ U \times \mathbb{R}^3}  g_0(\bar{v}) \mathbbm{1}_t^\varepsilon[\Phi] (\bar{x},\bar{v}) \, \mathrm{d}\bar{x} \, \mathrm{d}\bar{v}} & = \frac{1}{1- \int_{ U \times \mathbb{R}^3}  g_0(\bar{v}) (1-\mathbbm{1}_t^\varepsilon[\Phi] (\bar{x},\bar{v})) \, \mathrm{d}\bar{x} \, \mathrm{d}\bar{v}} \nonumber \\
& \leq 1+2\left( \int_{ U \times \mathbb{R}^3}  g_0(\bar{v}) (1-\mathbbm{1}_t^\varepsilon[\Phi] (\bar{x},\bar{v})) \, \mathrm{d}\bar{x} \, \mathrm{d}\bar{v} \right) = 1+2\eta_t^\varepsilon(\Phi).
\end{align}
This gives that,
\begin{align*}
 (1-\gamma(t)) & \left( \frac{\int_{\mathbb{S}^{2}}  \int_{\mathbb{R}^3}   g_0(\bar{v})[(v(\tau)-\bar{v})\cdot \nu]_+ \, \mathrm{d}\bar{v} \, \mathrm{d}\nu - \hat{C}_2(\varepsilon)}{\int_{ U \times \mathbb{R}^3}  g_0(\bar{v}) \mathbbm{1}_t^\varepsilon[\Phi] (\bar{x},\bar{v}) \, \mathrm{d}\bar{x} \, \mathrm{d}\bar{v}} \right) \\
& \leq \frac{\int_{\mathbb{S}^{2}}  \int_{\mathbb{R}^3}   g_0(\bar{v})[(v(\tau)-\bar{v})\cdot \nu]_+ \, \mathrm{d}\bar{v} \, \mathrm{d}\nu }{\int_{ U \times \mathbb{R}^3}  g_0(\bar{v}) \mathbbm{1}_t^\varepsilon[\Phi] (\bar{x},\bar{v}) \, \mathrm{d}\bar{x} \, \mathrm{d}\bar{v}}  \\
 & \leq (1+ 2\eta_t^\varepsilon(\Phi)) \int_{\mathbb{S}^{2}}  \int_{\mathbb{R}^3}   g_0(\bar{v})[(v(\tau)-\bar{v})\cdot \nu]_+ \, \mathrm{d}\bar{v} \, \mathrm{d}\nu.
\end{align*}
Finally giving that,
\[ (1-\gamma(t))\hat{L}_t(\Phi) \geq (1+2\eta_t^\varepsilon(\Phi))L(\Phi). \]
Substituting this into \eqref{eq-rtptcomp},
\begin{align} \label{eq-phatrode}
\partial_t \left( \hat{P}_t^\varepsilon(\Phi) -R_t^\varepsilon(\Phi)  \right) & = (1-\gamma(t))\hat{L}_t(\Phi)\hat{P}_t(\Phi) -L(\Phi)R_t^\varepsilon(\Phi)\nonumber \\
& \geq (1+2\eta_t^\varepsilon(\Phi))L(\Phi)\hat{P}_t(\Phi) -L(\Phi)R_t^\varepsilon(\Phi)\nonumber \\
& = (1+2\eta_t^\varepsilon(\Phi))L(\Phi)\left(\hat{P}_t(\Phi) -  R_t^\varepsilon(\Phi)  \right)  +2\eta_t^\varepsilon(\Phi)L(\Phi)  R_t^\varepsilon(\Phi).
\end{align}
For fixed $\Phi$ this is simply a $1d$ ODE in $t$. If $y:[\tau,\infty) \to \mathbb{R}$ satisfies,
\begin{equation*}
\begin{cases}
\frac{d}{dt}y(t) \geq a(t)y(t)+b(t), \\
y(\tau) = y_0.
\end{cases}
\end{equation*}
Then it follows by the variation of constants,
\begin{align*}
y(t) \geq \exp\left( \int_\tau^t a(s) \, \mathrm{d}s \right)y_0 + \int_\tau^t \exp \left( \int_s^t a(\sigma) \, \mathrm{d}\sigma  \right) b(s) \, \mathrm{d}s.
\end{align*}
Applying this to \eqref{eq-phatrode},
\begin{align} \label{eq-phatrcomp}
 \hat{P}_t^\varepsilon(\Phi) -R_t^\varepsilon(\Phi) & \geq \exp\left( \int_\tau^t (1+2\eta^\varepsilon_s(\Phi))L(\Phi) \, \mathrm{d}s \right)(\hat{P}_\tau^\varepsilon(\Phi) -R_\tau^\varepsilon(\Phi))\nonumber \\
 & \qquad  + \int_\tau^t \exp \left( \int_s^t (1+2\eta^\varepsilon_\sigma(\Phi))L(\Phi) \, \mathrm{d}\sigma  \right) 2\eta^\varepsilon_s(\Phi)L(\Phi)  R_s^\varepsilon(\Phi) \, \mathrm{d}s.
\end{align}
Recall the definition of $\mathbbm{1}_t^\varepsilon[\Phi]$ \eqref{eq-deff1phi} and note that it is non-increasing. Hence $\eta_t^\varepsilon(\Phi)$ is non-decreasing. So for $\tau\leq \sigma \leq t$, $\eta^\varepsilon_\sigma(\Phi) \leq \eta_t^\varepsilon(\Phi)$. Recalling that $L(\Phi)$ is non-positive, \eqref{eq-phatrcomp} becomes,
\begin{align}
\hat{P}_t^\varepsilon(\Phi) -R_t^\varepsilon(\Phi) & \geq \exp\left( \int_\tau^t (1+2\eta_s^\varepsilon(\Phi))L(\Phi) \, \mathrm{d}s \right)(\hat{P}_\tau^\varepsilon(\Phi) -R_\tau^\varepsilon(\Phi))\nonumber \\
 & \qquad  +2\eta_t^\varepsilon(\Phi) \int_\tau^t \exp \left( \int_s^t (1+2\eta^\varepsilon_\sigma(\Phi))L(\Phi) \, \mathrm{d}\sigma  \right) L(\Phi)  R_s^\varepsilon(\Phi) \, \mathrm{d}s \nonumber \\
& \geq \exp\left( \int_\tau^t (1+2\eta_s^\varepsilon(\Phi))L(\Phi) \, \mathrm{d}s \right)(\hat{P}_\tau^\varepsilon(\Phi) -R_\tau^\varepsilon(\Phi))\nonumber \\
 & \qquad  +2\eta_t^\varepsilon(\Phi)L(\Phi) \int_\tau^t \exp \Big( (1+2\eta_s^\varepsilon(\Phi)) (t-s)  \,  L(\Phi) \Big)  R_s^\varepsilon(\Phi) \, \mathrm{d}s
 \label{eq-phatrcomp3} .
\end{align}
Finally for $t>\tau$,
\begin{align*}
\partial_t R_t^\varepsilon(\Phi) & =\zeta(\varepsilon) \partial_t P_t(\Phi) = \zeta(\varepsilon)P_t(\Phi)L(\Phi) = R_t^\varepsilon(\Phi)L(\Phi).
\end{align*}
Hence for $\tau \leq s \leq t$,
\[ R_t^\varepsilon(\Phi) = \exp \left( (t-s) L (\Phi)  \right) R_s^\varepsilon(\Phi). \]
Which implies,
\begin{equation} \label{eq-rexpdecay}
R_s^\varepsilon(\Phi) = \exp \left( - (t-s)L(\Phi)  \right) R_t^\varepsilon(\Phi).
\end{equation}
Substituting this into \eqref{eq-phatrcomp3},
\begin{align*}
\hat{P}_t^\varepsilon(\Phi) -R_t^\varepsilon(\Phi)  & \geq \exp\left( \int_\tau^t (1+2\eta_s^\varepsilon(\Phi))L(\Phi) \, \mathrm{d}s \right)(\hat{P}_\tau^\varepsilon(\Phi) -R_\tau^\varepsilon(\Phi))\nonumber \\
 & \qquad + 2\eta_t^\varepsilon(\Phi)L(\Phi) \int_\tau^t \exp \Big( (1+2\eta_s^\varepsilon(\Phi)) (t-s) L(\Phi)  \Big)    \\
 & \qquad  \times  \exp \left( - (t-s)L(\Phi) \right) R_t^\varepsilon(\Phi) \, \mathrm{d}s \\
 & \geq \exp\left( \int_\tau^t (1+2\eta_s^\varepsilon(\Phi))L(\Phi) \, \mathrm{d}s \right)(\hat{P}_\tau^\varepsilon(\Phi) -R_\tau^\varepsilon(\Phi))\nonumber \\
 & \qquad  +2\eta_t^\varepsilon(\Phi) L(\Phi) R_t^\varepsilon(\Phi) \int_\tau^t \exp \left(2\eta_s^\varepsilon(\Phi) (t-s) L(\Phi) \right)   \, \mathrm{d}s.
\end{align*}
This completes the proof of the lemma.
\end{proof}
\begin{lem} \label{lem-rho0bound}
For $\Phi \in \mathcal{G}(\varepsilon)$ and $t \geq \tau$,
\[   2\eta_t^\varepsilon(\Phi) L(\Phi)  \int_\tau^t \exp \left(2\eta_s^\varepsilon(\Phi) (t-s) L(\Phi) \right)   \, \mathrm{d}s \geq -\rho^{\varepsilon,0}_t(\Phi). \]

\end{lem}
\begin{proof}
Since $L(\Phi) \leq 0$,
\[    \int_\tau^t \exp \left(2\eta_s^\varepsilon(\Phi) (t-s) L(\Phi) \right)   \, \mathrm{d}s \leq t-\tau \leq t. \]
Recalling \eqref{eq-cdeff},
\[ -L(\Phi) \leq \frac{C(\Phi)}{2}. \]
Hence combining these,
\[   2 L(\Phi)  \int_\tau^t \exp \left(2\eta_s^\varepsilon(\Phi) (t-s) L(\Phi) \right)   \, \mathrm{d}s \geq -C(\Phi)t. \]
Multiplying both sides by $\eta_t^\varepsilon(\Phi)$ gives the required identity.
\end{proof}
\begin{lem} \label{lem-gammavolbound}
For $\varepsilon$ sufficiently small, for any $\Phi \in \mathcal{G}(\varepsilon)$ and any $t \in [0,T]$,
\[ 1- \frac{1-\gamma(t)}{1-|\Delta|} \leq \varepsilon .\]
\end{lem}
\begin{proof}
We show that for $\varepsilon$ sufficiently small,
\[ \frac{1}{\varepsilon}\left(1- \frac{1-\gamma(t)}{1-|\Delta|}\right) \leq 1.  \]
Note,
\begin{align*}
|\Delta| & = \pi \varepsilon^2 \int_0^\tau |v(s)-v'| \, \mathrm{d}s - \hat{C}\varepsilon^3 \leq  \pi \varepsilon^2 \int_0^\tau |v(s)-v'| \, \mathrm{d}s
\leq \pi  \varepsilon^2  2T V(\varepsilon).
\end{align*}
Recalling \eqref{eq-massump},
\begin{align*}
1- \frac{1-\gamma(t)}{1-|\Delta|} & = \frac{1-|\Delta|}{1-|\Delta|}- \frac{1-\gamma(t)}{1-|\Delta|} = \frac{\gamma(t)-|\Delta|}{1-|\Delta|} \\
& \leq \frac{\gamma(t)}{1-|\Delta|} \leq \frac{\varepsilon^2n(\Phi)}{1-2\varepsilon^2 \pi TV(\varepsilon) } \leq \frac{\varepsilon^2M(\varepsilon)}{1-2\varepsilon^2 \pi TV(\varepsilon) }  \leq \frac{\varepsilon^{3/2}}{1-2\varepsilon^2 \pi TV(\varepsilon) }.
\end{align*}
Therefore
\[  \frac{1}{\varepsilon}\left(1- \frac{1-\gamma(t)}{1-|\Delta|}\right) \leq \frac{\varepsilon^{1/2}}{1-2\varepsilon^2 \pi TV(\varepsilon) }, \]
which, by recalling \eqref{eq-vassump2}, converges to zero as $\varepsilon$ converges to zero. Hence for $\varepsilon$ sufficiently small the right hand side is less than $1$.

\end{proof}
\begin{proof}[Proof of proposition~\ref{prop-ptphatcomp}]
We prove by induction on the degree of $\Phi$. Firstly we show that the proposition holds for $\Phi \in \mathcal{T}_0 \cap \mathcal{G}(\varepsilon)$. Now if $\Phi \in \mathcal{T}_0 \cap \mathcal{G}(\varepsilon)$ it follows that $\tau =0$ and hence,
\begin{align*}
\hat{P}_\tau^\varepsilon(\Phi) & = \hat{P}_0^\varepsilon(\Phi) = \zeta(\varepsilon)f_0(x_0,v_0) = \zeta(\varepsilon)P_0(\Phi)  = R_\tau^\varepsilon(\Phi).
\end{align*}
By lemma~\ref{lem-pphatcomphelp} and \ref{lem-rho0bound} for $t \geq 0$,
\begin{align*}
\hat{P}_t^\varepsilon(\Phi) - R_t^\varepsilon(\Phi) & \geq  2\eta_t^\varepsilon(\Phi) L(\Phi) R_t^\varepsilon(\Phi) \int_\tau^t \exp \left(2\eta_s^\varepsilon(\Phi) (t-s) L(\Phi) \right)   \, \mathrm{d}s \geq - \rho^{\varepsilon,0}_t(\Phi)R_t^\varepsilon(\Phi) \\
& = - \hat{\rho}^{\varepsilon}_t(\Phi)R_t^\varepsilon(\Phi).
\end{align*}
Proving the proposition in the base case. Now suppose that the proposition holds true for all trees $\Phi \in \mathcal{T}_{k-1}\cap \mathcal{G}(\varepsilon)$ for some $k \geq 1$ and let $\Psi \in \mathcal{T}_{k}\cap \mathcal{G}(\varepsilon)$. For $t<\tau$ the proposition holds trivially, so we consider $t\geq \tau$.
Now recall that,
\[ \hat{P}_\tau^\varepsilon(\Psi) = \frac{1-\gamma(t)}{1-|\Delta|}\hat{P}_\tau^\varepsilon(\bar{\Psi})g_0(v')[(v(\tau^-)-v')\cdot \nu]_+ , \]
and,
\[ R_\tau^\varepsilon(\Psi) = R_\tau^\varepsilon(\bar{\Psi})g_0(v')[(v(\tau^-)-v')\cdot \nu]_+ . \]
Further since $\bar{\Psi} \in \mathcal{T}_{k-1}$ we know by our inductive assumption that the proposition holds for $\bar{\Psi}$, which implies.
\[  \hat{P}^\varepsilon_t(\bar{\Psi}) \geq R^\varepsilon_t(\bar{\Psi}) - \hat{\rho}^\varepsilon_t(\bar{\Psi})R_t^\varepsilon(\bar{\Psi}). \]
Hence by the estimate in lemma~\ref{lem-gammavolbound} for $\varepsilon$ sufficiently small,
\begin{align} \label{eq-ptauindstep}
 \hat{P}_\tau^\varepsilon(\Psi) - R_\tau^\varepsilon(\Psi) & = g_0(v')[(v(\tau^-)-v')\cdot \nu]_+ \left(\frac{1-\gamma(t)}{1-|\Delta|}\hat{P}_\tau^\varepsilon(\bar{\Psi}) - R_\tau^\varepsilon(\bar{\Psi}) \right) \nonumber \\
 & \geq g_0(v')[(v(\tau^-)-v')\cdot \nu]_+ \left((1-\varepsilon)\hat{P}_\tau^\varepsilon(\bar{\Psi}) - R_\tau^\varepsilon(\bar{\Psi}) \right) \nonumber \\
 & \geq  g_0(v')[(v(\tau^-)-v')\cdot \nu]_+ \left((1-\varepsilon)(R^\varepsilon_\tau(\bar{\Psi}) - \hat{\rho}_t(\bar{\Psi})R_\tau^\varepsilon(\bar{\Psi})) - R_\tau^\varepsilon(\bar{\Psi}) \right) \nonumber \\
 & = g_0(v')[(v(\tau^-)-v')\cdot \nu]_+ R_\tau^\varepsilon(\bar{\Psi})  \left(1-\varepsilon - (1-\varepsilon)\hat{\rho}^\varepsilon_\tau(\bar{\Psi}) -1  \right)\nonumber \\
 & = g_0(v')[(v(\tau^-)-v')\cdot \nu]_+ R_\tau^\varepsilon(\bar{\Psi})  \left(-\varepsilon - (1-\varepsilon)\hat{\rho}^\varepsilon_\tau(\bar{\Psi})   \right) = R_\tau^\varepsilon(\Psi)  \left(-\varepsilon -(1-\varepsilon) \hat{\rho}^\varepsilon_\tau(\bar{\Psi})   \right) .
\end{align}
Since the trajectory of the root particle of $\Psi$ up to time $\tau$ is equal to the trajectory of the root particle of $\bar{\Psi}$ up to time $\tau$ and recalling that for any $\Phi$, $\eta_t^\varepsilon(\Phi)$ is non-decreasing in $t$, it follows,
\begin{align*}
\eta_\tau^\varepsilon(\bar{\Psi}) & =\int_{U \times \mathbb{R}^3} g_0(\bar{v})(1- \mathbbm{1}_\tau^\varepsilon[\bar{\Psi}](\bar{x},\bar{v})) \, \mathrm{d}\bar{x} \, \mathrm{d}\bar{v} \nonumber \\
& = \int_{U \times \mathbb{R}^3} g_0(\bar{v})(1- \mathbbm{1}_\tau^\varepsilon[\Psi](\bar{x},\bar{v})) \, \mathrm{d}\bar{x} \, \mathrm{d}\bar{v},
 = \eta_\tau^\varepsilon(\Psi) \leq \eta_t^\varepsilon(\Psi),
\end{align*}
and recalling \eqref{eq-cdeff},
\begin{align*}
C(\bar{\Psi}) \leq C(\Psi).
\end{align*}
Hence,
\begin{align*}
\rho^{\varepsilon,0}_\tau(\bar{\Psi}) & = \eta^\varepsilon_\tau(\bar{\Psi}) C(\bar{\Psi})\tau \leq \eta_t^\varepsilon(\Psi)C(\Psi) t= \rho^{\varepsilon,0}_t(\Psi).
\end{align*}
Recalling \eqref{eq-rhoalt}, this implies,
\begin{align*}
\hat{\rho}^{\varepsilon}_\tau(\bar{\Psi}) & = \rho^{\varepsilon,k-1}_\tau(\bar{\Psi})  \leq \rho^{\varepsilon,k-1}_t(\Psi).
\end{align*}
Substituting this into \eqref{eq-ptauindstep},
\begin{align*} \label{eq-phatrindstep2}
 \hat{P}_\tau^\varepsilon(\Psi) - R_\tau^\varepsilon(\Psi) & \geq R_\tau^\varepsilon(\Psi)  \left(-\varepsilon -(1-\varepsilon) \hat{\rho}^\varepsilon_\tau(\bar{\Psi})   \right) \nonumber \\
 & \geq R_\tau^\varepsilon(\Psi)  \left(-\varepsilon -(1-\varepsilon) \rho^{\varepsilon,k-1}_t(\Psi)  \right) \nonumber \\
 & = -R_\tau^\varepsilon(\Psi)  \left(\varepsilon +(1-\varepsilon) \rho^{\varepsilon,k-1}_t(\Psi)  \right) .
\end{align*}
Recalling \eqref{eq-rexpdecay},
\begin{align*}
\exp & \left( L(\Psi) \int_\tau^t(1+2\eta_s^\varepsilon(\Psi))   \, \mathrm{d}s \right) (\hat{P}_\tau^\varepsilon(\Psi) - R_\tau^\varepsilon(\Psi)) \\
& \geq - \exp\left( L(\Psi) \int_\tau^t (1+2\eta_s^\varepsilon(\Psi)) \, \mathrm{d}s \right)R_\tau^\varepsilon(\Psi)  \left(\varepsilon +(1-\varepsilon) \rho^{\varepsilon,k-1}_t(\Psi)  \right) \\
& = - \exp\left(L(\Psi) \int_\tau^t 2\eta_s^\varepsilon(\Psi)   \, \mathrm{d}s \right)R_t^\varepsilon(\Psi)  \left(\varepsilon +(1-\varepsilon) \rho^{\varepsilon,k-1}_t(\Psi)  \right) \\
& \geq - R_t^\varepsilon(\Psi)  \left(\varepsilon +(1-\varepsilon) \rho^{\varepsilon,k-1}_t(\Psi)  \right).
\end{align*}
Recalling lemma~\ref{lem-pphatcomphelp} and \ref{lem-rho0bound} this gives,
\begin{align*}
\hat{P}_t^\varepsilon(\Psi) - R_t^\varepsilon(\Psi) & \geq \exp\left(  \int_\tau^t L(\Psi)(1+2\eta_s^\varepsilon(\Psi)) \, \mathrm{d}s \right)(\hat{P}_\tau^\varepsilon(\Psi) - R_\tau^\varepsilon(\Psi)) \\
&  \qquad +2\eta_t^\varepsilon(\Psi) L(\Psi) R_t^\varepsilon(\Psi) \int_\tau^t \exp \left(2\eta_s^\varepsilon(\Psi) (t-s) L(\Psi) \right)   \, \mathrm{d}s \\
& \geq - R_t^\varepsilon(\Psi)  \left(\varepsilon +(1-\varepsilon) \rho^{\varepsilon,k-1}_t(\Psi)  \right) - \rho_t^{\varepsilon,0}(\Psi)R_t^\varepsilon(\Psi) \\
& \geq -R_t^\varepsilon(\Psi) \left( \varepsilon +(1-\varepsilon) \rho^{\varepsilon,k-1}_t(\Psi) +\rho_t^{\varepsilon,0}(\Psi)  \right) \\
& =  -R_t^\varepsilon(\Psi) \rho_t^{\varepsilon,k}(\Psi)  = -R_t^\varepsilon(\Psi) \hat{\rho}_t^{\varepsilon}(\Psi).
\end{align*}
This completes the proof of the inductive step which concludes the proof of the proposition.
\end{proof}
\begin{prop} \label{prop-goodtreefull} Good trees have full measure in the sense that
\[ \lim_{\varepsilon \to 0}P_t(\mathcal{MT} \setminus \mathcal{G}(\varepsilon)) =0. \]
\end{prop}
\begin{proof}
Firstly we prove that $\mathcal{G}(0)$ is of measure 1. To this aim recall the definition of $R(\varepsilon)$ from definition~\ref{def-recollisionfree} and the definition of $\mathcal{T}_j$ \eqref{eq-tjdeff}. Define for $j \in \mathbb{N} \cup \{0 \}$,
\[ R_j(\varepsilon) := \mathcal{T}_j \cap R(\varepsilon).  \]
Now note that $\mathcal{T}_0 \setminus R_0(\varepsilon)$ and $\mathcal{T}_1 \setminus R_1(\varepsilon)$ are empty, since it is not possible for a tree in $\mathcal{T}_0$ or $\mathcal{T}_1$ to have a re-collision. For a general $\Phi \in \mathcal{T}_2 $, $\Phi$ is of the form,
\[ \Phi = \big( (x_0,v_0),(s_1,\nu_1,v_1),(s_2,\nu_2,v_2) \big). \]
If we require that $\Phi \in \mathcal{T}_2 \setminus R_2(0)$ then there must be a re-collision and since the background particles do not change velocity this requires that $v_1=v_2$. Hence $\mathcal{T}_2 \setminus R_2(0)$ is a submanifold of $\mathcal{T}_2$ with,
\[ \dim(\mathcal{T}_2 \setminus R_2(0)) < \dim(\mathcal{T}_2).  \]
This argument holds true for all $j\geq 2$ and hence, because $\mathcal{MT} = \cup _{j=0}^{\infty} \mathcal{T}_j$ we see that,
\[ P_t(\mathcal{MT}\setminus R(0))=1.\]
Therefore it follows, since the other requirements on $\mathcal{G}(0)$ are clear, that
$P_t(\mathcal{MT}\setminus \mathcal{G}(0))=1$, and hence also that $P_t( \mathcal{G}(0))=1$.

Since $\mathcal{G}(\varepsilon)$ is increasing as $\varepsilon$ decreases and $\lim_{\varepsilon \to 0}\mathcal{G}(\varepsilon)=\mathcal{G}(0)$ it follows by the dominated convergence theorem that,
\[ \lim_{\varepsilon \to 0}P_t( \mathcal{G}(\varepsilon)) =P_t(\mathcal{G}(0)) =1.  \]
Hence
\[ \lim_{\varepsilon \to 0}P_t( \mathcal{MT} \setminus \mathcal{G}(\varepsilon)) =0, \]
as required.
\end{proof}
\begin{lem} \label{lem-convbounds}
Recall $\beta$ \eqref{eq-beta}. For $\varepsilon>0$, $\Phi \in \mathcal{G}(\varepsilon)$ there exists constants $C_1,C_2> 0$ such that,
\begin{align}
\eta_t^\varepsilon(\Phi)& \leq C_1\varepsilon^2(\beta+V(\varepsilon))\label{eq-convbounds1}, \\
C(\Phi) & \leq C_2 (\beta+V(\varepsilon)). \label{eq-convbounds2} \end{align}
\end{lem}
\begin{proof}
Firstly by \eqref{eq-etabound},
\begin{align*}
\eta_t^\varepsilon(\Phi) & = \int_{U \times \mathbb{R}^3} g_0(\bar{v})(1- \mathbbm{1}_t^\varepsilon[\Phi](\bar{x},\bar{v})) \, \mathrm{d}\bar{x} \, \mathrm{d}\bar{v} \leq \pi \varepsilon^2 T (\beta+V(\varepsilon)).
\end{align*}
So take $C_1:= \pi T$ proving \eqref{eq-convbounds1}. Next note,
\begin{align*}
 \int_{\mathbb{S}^{2}} \int_{\mathbb{R}^3} g_0(\bar{v})[(v(t)-\bar{v})\cdot \nu ]_+ \, \mathrm{d}\bar{v} \, \mathrm{d}\nu & \leq \pi  \int_{\mathbb{R}^3} g_0(\bar{v})|v(t)-\bar{v}| \, \mathrm{d}\bar{v} \\
& \leq \pi \int_{\mathbb{R}^3} g_0(\bar{v})(\mathcal{V}(\Phi) + |\bar{v}|) \, \mathrm{d}\bar{v} \leq \pi (V(\varepsilon)+\beta).
\end{align*}
Hence by \eqref{eq-cdeff},
\begin{align*}
C(\Phi) & \leq 2 \pi (V(\varepsilon)+\beta).
\end{align*}
so take $C_2 :=2 \pi$ which proves \eqref{eq-convbounds2}.
\end{proof}
\begin{lem} \label{lem-rhokbound}
For any $\delta > 0$, there exists a  $\varepsilon'>0$ such that for $0<\varepsilon < \varepsilon'$ and for any $\Phi \in \mathcal{G}(\varepsilon)$,
\[ \hat{\rho}_t^{\varepsilon}(\Phi) < \delta. \]
\end{lem}
\begin{proof}
Fix $\delta > 0$. Firstly by the above lemma,
\begin{align*}
\rho_t^{\varepsilon,0}(\Phi) & = \eta_t(\Phi)C(\Phi)t  \leq C_1C_2T \varepsilon^2(\beta+V(\varepsilon))^2.
\end{align*}
Recalling \eqref{eq-vassump2} there exists an $\varepsilon_1>0$ such that for $0<\varepsilon<\varepsilon_1$,  $$\rho_t^{\varepsilon,0}(\Phi)<\frac{\delta}{3}.$$
Further there exists an $\varepsilon_2>0$ such that for $0<\varepsilon<\varepsilon_2$,
\[  \frac{1}{\sqrt[]{\varepsilon}}\rho_t^{\varepsilon,0}(\Phi) \leq C_1C_2T \varepsilon^{3/2}(\beta+V(\varepsilon))^2 < \frac{\delta}{3}.  \]
Finally there exists an $\varepsilon_3>0$ such that for $0<\varepsilon<\varepsilon_3$, $\sqrt[]{\varepsilon}<\frac{\delta}{3}$.
Hence take $\varepsilon'=\min\{\varepsilon_1,\varepsilon_2,\varepsilon_3,1\}$ then for any $\Phi \in \mathcal{G}(\varepsilon)$,
\begin{align*}
\hat{\rho}^{\varepsilon}_t(\Phi) & = \rho^{\varepsilon,n(\Phi)}_t(\Phi)  =  (1-\varepsilon)^{n(\Phi)} \rho^{\varepsilon,0}_t (\Phi) + (\rho^{\varepsilon,0}_t (\Phi) + \varepsilon) \sum_{j=1}^{n(\Phi)}(1-\varepsilon)^{n(\Phi)-j} \\
& \leq \rho^{\varepsilon,0}_t (\Phi) + (\rho^{\varepsilon,0}_t (\Phi) + \varepsilon)\times  n(\Phi)
 \leq \rho^{\varepsilon,0}_t (\Phi) + M(\varepsilon) (\rho^{\varepsilon,0}_t (\Phi) + \varepsilon) \\
& = \rho^{\varepsilon,0}_t (\Phi) + M(\varepsilon) \rho^{\varepsilon,0}_t (\Phi) + M(\varepsilon)\varepsilon  \leq \rho^{\varepsilon,0}_t (\Phi) + \frac{1}{\sqrt[]{\varepsilon}} \rho^{\varepsilon,0}_t (\Phi) + \sqrt[]{\varepsilon} < \delta.
\end{align*}

\end{proof}

\begin{thm} \label{thm-ptphatcomp}
Uniformly for $t \in [0,T]$,
\[ \lim_{\varepsilon \to 0} \|P_t - \hat{P}^\varepsilon	_t\|_{TV}=0. \]
\end{thm}
\begin{proof}
Let $\delta > 0$ and $S \subset \mathcal{MT}$ then,
\begin{align*}
P_t(S) - \hat{P}^\varepsilon	_t(S) &= P_t(S \cap \mathcal{G}(\varepsilon))+ P_t(S \setminus \mathcal{G}(\varepsilon)) - \hat{P}^\varepsilon	_t(S \cap \mathcal{G}(\varepsilon)) - \hat{P}^\varepsilon	 _t(S \setminus \mathcal{G}(\varepsilon)) \nonumber \\
& \leq P_t(S \cap \mathcal{G}(\varepsilon))+ P_t(S\setminus \mathcal{G}(\varepsilon)) - \hat{P}^\varepsilon	_t(S \cap \mathcal{G}(\varepsilon)).
\end{align*}
By proposition~\ref{prop-goodtreefull} for $\varepsilon$ sufficiently small,
\begin{align*}
P_t(S \setminus \mathcal{G}(\varepsilon))  \leq P_t(\mathcal{MT} \setminus \mathcal{G}(\varepsilon))< \frac{\delta}{3}.
\end{align*}
Hence,
\begin{equation} \label{eq-ptphatcomp1}
P_t(S) - \hat{P}^\varepsilon_t(S) < P_t(S \cap \mathcal{G}(\varepsilon))  - \hat{P}^\varepsilon_t(S \cap \mathcal{G}(\varepsilon)) +\frac{\delta}{3}.
\end{equation}
Recall the definition of $\zeta(\varepsilon)$ in \eqref{eq-zetadeff}. It is clear that this implies
\begin{equation} \label{eq-zetaupbound}
\zeta(\varepsilon) \leq 1.
\end{equation}
Hence by the above lemma, for $\varepsilon$ sufficiently small and $\Phi \in \mathcal{G}(\varepsilon)$,
\[ \zeta(\varepsilon) \hat{\rho}^{\varepsilon}_t(\Phi)\leq \hat{\rho}^{\varepsilon}_t(\Phi) < \frac{\delta}{3}. \]
Further by proposition~\ref{prop-ptphatcomp}, for $\Phi \in \mathcal{G}(\varepsilon)$,
\begin{equation} \label{eq-minushatpbound}
 -\hat{P}^\varepsilon_t(\Phi) \leq - R^\varepsilon_t(\Phi) +\hat{\rho}^\varepsilon_t(\Phi)R_t^\varepsilon(\Phi) .
\end{equation}
The Binomial inequality states that for $x \geq -1$ and $ N \in \mathbb{N}$,
\[ (1+x)^N \geq 1+Nx. \]
Hence for $\varepsilon>0$ such that $\frac{4}{3}\pi \varepsilon^3 \leq 1$ we apply this to $\zeta(\varepsilon)$ recalling \eqref{eq-boltzgrad},
\begin{align} \label{eq-zetalowbound}
\zeta(\varepsilon) & = (1-\frac{4}{3}\pi \varepsilon^3)^N  \geq 1-N\frac{4}{3}\pi \varepsilon^3  = 1-\frac{4}{3}\pi\varepsilon.
\end{align}
Hence for $\varepsilon$ sufficiently small \eqref{eq-minushatpbound} gives,
\begin{align*}
 P_t(\Phi)-\hat{P}^\varepsilon_t(\Phi) & \leq P_t(\Phi) - R^\varepsilon_t(\Phi) + \hat{\rho}^\varepsilon_t(\Phi)R_t^\varepsilon(\Phi) \nonumber \\
 &  = P_t(\Phi) - \zeta(\varepsilon)P_t(\Phi) + \hat{\rho}^\varepsilon_t(\Phi)\zeta(\varepsilon)P_t(\Phi) \leq \frac{4}{3}\pi\varepsilon P_t(\Phi) + \frac{\delta}{3} P_t(\Phi)  \leq \frac{2\delta}{3} P_t(\Phi).
\end{align*}
This holds for all $\Phi \in \mathcal{G}(\varepsilon)$ with $\varepsilon$ sufficiently small hence
\begin{align} \label{eq-ptphatcomp2}
 P_t(S \cap \mathcal{G}(\varepsilon))-\hat{P}^\varepsilon_t(S \cap \mathcal{G}(\varepsilon))  \leq \frac{2\delta}{3} P_t(S \cap \mathcal{G}(\varepsilon)) \leq \frac{2\delta}{3}.
\end{align}
By substituting \eqref{eq-ptphatcomp2} into \eqref{eq-ptphatcomp1}, for $\varepsilon$ sufficiently small,
\begin{equation} \label{eq-pphatcompdeltalower}
P_t(S) - \hat{P}^\varepsilon_t(S) < \delta.
\end{equation}
Since $\varepsilon$ did not depend on $S$ this holds true for every $S \subset \mathcal{MT}$. Hence for any $S \subset \mathcal{MT}$,
\begin{align*}
\hat{P}^\varepsilon_t(S) -  P_t(S) & = (1-\hat{P}^\varepsilon_t(\mathcal{MT}\setminus S) ) - (1-P_t(\mathcal{MT}\setminus S))
& = P_t(\mathcal{MT}\setminus S) - \hat{P}^\varepsilon_t(\mathcal{MT}\setminus S) < \delta.
\end{align*}
This together with \eqref{eq-pphatcompdeltalower} gives that, for $\varepsilon$ sufficiently small, for any $S \subset \mathcal{MT}$,
\[ |P_t(S) - \hat{P}^\varepsilon_t(S) |< \delta, \]
which completes the proof of the theorem.  \end{proof}

We can now prove the main theorem of the paper, theorem~\ref{thm-main}, which follows from the above theorem.
\begin{proof}[Proof of theorem~\ref{thm-main}.]
Recall for $\Omega \subset U\times \mathbb{R}^3$, $$S_t(\Omega):=\{ \Phi \in \mathcal{MT} : (x(t),v(t)) \in \Omega \}.$$
By theorem~\ref{thm-ideq},
\[ \int_\Omega f_t(x,v) \, \mathrm{d}x \, \mathrm{d}v = \int_{S_t(\Omega)} P_t(\Phi) \, \mathrm{d}\Phi = P_t(S_t(\Omega)) .\]
Also by definition of $\hat{P}_t^\varepsilon$,
\[ \int_\Omega \hat{f}^N_t(x,v) \, \mathrm{d}x \, \mathrm{d}v = \int_{S_t(\Omega)} \hat{P}_t^\varepsilon(\Phi) \, \mathrm{d}\Phi = \hat{P}_t^\varepsilon(S_t(\Omega))  .\]
Hence by theorem~\ref{thm-ptphatcomp},
\begin{align*}
\lim_{N \to \infty} \sup_{\Omega \subset U \times \mathbb{R}^3} & | \int_\Omega \hat{f}^N_t(x,v) - f_t(x,v) \, \mathrm{d}x \, \mathrm{d}v  |
= \lim_{\varepsilon \to 0} \sup_{\Omega \subset U \times \mathbb{R}^3} |P_t(S_t(\Omega)) - \hat{P}_t^\varepsilon(S_t(\Omega))| = 0,
\end{align*}
which completes the proof.
\end{proof}

\section{Proof of auxiliary results}
\label{sec:aux}

\subsection{Particle dynamics}

\begin{proof}[Proof of Prop \ref{prop:wellposedHS}]
The dynamics become undefined if there is instantaneously more than one background particle colliding with the tagged particle or if the tagged particle experiences an infinite number of collisions in finite time. We adapt a similar proof for the full hard-spheres dynamics from  \cite[Prop 4.1.1]{saintraymond13}.

Let $R>0$ and $\delta<\varepsilon/2$ such that there exists a $K\in \mathbb{N}$ with $T=K\delta$. Denote the ball of radius $R$ about $x$ in $\mathbb{R}^3$ by $B_R(x)$. For the initial position of the tagged particle $(x_0,v_0)\in U \times B_R(0)$ fixed define $I(x_0,v_0) \subset (U \times \mathbb{R}^3)^N $ by,
\begin{align*}
I(x_0,v_0) & :=  \{(x_1,v_1),\dots,(x_N,v_N)   \in (U \times B_R(0))^N : \textrm{ the tagged particle collides} \\
& \qquad \textrm{ with at least two background particles in the time interval } [0,\delta] \}.
\end{align*}
We bound the volume of this set. Firstly define
\[ I^1(x_0,v_0): = \{ (x_1,v_1) \in U \times B_R(0): \varepsilon \leq |x_0-x_1| \leq \varepsilon+2R\delta  \}.  \]
It can be seen that for some $C$,
\[ |I^1(x_0,v_0)| \leq CR^3 \times (2R\delta)^3. \]
Since $I(x_0,v_0)$ is a subset of,
\begin{align*}
 \{ (x_1,v_1) & ,\dots,(x_N,v_N)   \in (U \times B_R(0))^N : \exists \, 1 \leq i < j \leq N \textrm{ such that }\\
 & \qquad \varepsilon \leq |x_0-x_i| \leq \varepsilon+2R\delta \textrm{ and } \varepsilon \leq |x_0-x_j| \leq \varepsilon+2R\delta  \}
\end{align*}
The above estimate gives, for some constant $C=C(N,\varepsilon)$,
\begin{align*}
|I(x_0,v_0)| & \leq CR^{3(N-2)} \times (R^6\delta^3)^2 \\
& \leq CR^{3(N+2)} \delta^6.
\end{align*}
Hence if we define,
\[ I:= \cup \{I(x_0,v_0) : (x_0,v_0) \in U \times B_R(0)\}, \]
it follows,
\[ |I| \leq CR^{3(N+3)} \delta^6.  \]
Hence there exists a subset $I_0(\delta,R)$ of measure at most $CR^{3(N+3)} \delta^6$ such that for any initial configuration in $(U\times B_R(0))^{N+1} \setminus I_0(\delta,R)$ the tagged particle experiences at most one collision in $[0,\delta]$.

Now consider the system at time $\delta$. Since all particles had initial velocity in $B_R(0)$ and the tagged particle had at most one collision in time $[0,\delta]$ the velocity of the tagged particle at time $\delta$ is in $B_{2R}(0)$. By the same arguments above there exists a set $I_1(\delta,R)$ of measure at most $CR^{3(N+3)} \delta^6$ for some new constant $C$ such that for any initial configuration in $(U\times B_R(0))^{N+1} \setminus I_0(\delta,R)\cup I_1(\delta,R)$ the tagged particle experiences at most one collision in $[0,\delta]$ and at most one collision in $[\delta,2\delta]$ and thus the dynamics are well defined up to $2\delta$.

Continue this process $K$ times defining the set,
\[ I(\delta,R) := \cup_{j=0}^{K-1}  I_j(\delta,R), \]
which has measure at most $CR^{3(N+3)} \delta^6$ for some new constant $C$ and such that for any initial configuration in $(U\times B_R(0))^{N+1} \setminus I(\delta,R)$ the tagged particle has at most one collision per time interval $[j\delta, (j+1)\delta ]$ and hence the dynamics are well defined up to time $T$. Defining,
\[ I(T,R):= \cap_{\delta > 0 } I(\delta,R), \]
if follows $I(T,R)$ is of measure zero and for any any initial configuration in $(U\times B_R(0))^{N+1} \setminus I(T,R)$ the dynamics are well defined up to time $T$.

Finally take,
\[ I := \cup_{R \in \mathbb{N} } I(T,R) \]
and note that $I$ is a countable union of measure zero sets and for any initial configuration in $(U\times \mathbb{R}^3) \setminus I$ the dynamics are well defined up to time $T$.
\end{proof}

\subsection{Semigroup theory}

\begin{lem} \label{lem-tsemi}
$A$ is the generator of the substochastic semigroup  (c.f.  \cite[Section 10.2]{arlotti06}) $T(t):L^1(U \times \mathbb{R}^3) \to L^1(U \times \mathbb{R}^3)$ given by,
\begin{equation}\label{eq-tsemi}
(T(t)f)(x,v) := \exp \left( -t \int_{\mathbb{S}^2} \int_{\mathbb{R}^3} g_0(\bar{v}) [(v-\bar{v})\cdot \nu]_+ \, \mathrm{d}\bar{v} \, \mathrm{d}\nu  \right) f(x-tv,v). \end{equation}
\end{lem}
\begin{proof}
We seek to apply theorem~10.4 of \cite{arlotti06}. Conditions (A1), (A2) trivially hold since $F=0$ in our situation. As for (A3),
\begin{equation} \label{eq-arlottinu}
\nu(x,v)= \int_{\mathbb{S}^2} \int_{\mathbb{R}^3} g_0(\bar{v}) [(v-\bar{v})\cdot \nu]_+ \, \mathrm{d}\bar{v} \, \mathrm{d}\nu.
\end{equation}
This is locally integrable so (A3) holds. Hence we can apply the theorem. In our case,
\[ \varphi(x,v,t,s)=x-(t+s)v. \]
So the semigroup is given by
\[ (T(t)f)(x,v) := \exp \left( -t \int_{\mathbb{S}^2} \int_{\mathbb{R}^3} g_0(\bar{v}) [(v-\bar{v})\cdot \nu]_+ \, \mathrm{d}\bar{v} \, \mathrm{d}\nu  \right) f(x-tv,v), \]
as required.
\end{proof}

\begin{rmk}
For $v,v_* \in \mathbb{R}^3, v \neq v_*$ define the Boltzmann kernel $k$ by,
\begin{equation} \label{eq-k}
k(v,v_*) := \frac{1}{|v-v_*|} \int_{E_{v,v_*}} g_0(w) \, \mathrm{d}w,
\end{equation}
where $E_{v,v_*} = \{ w \in \mathbb{R}^3 : w\cdot (v-v_*) = v\cdot (v-v_*)  \}$.
By the use of Carleman's representation, first described in \cite{carleman57} (see also \cite[Section 3]{bisi14}),
\begin{align} \label{eq-bdeffalt}
(Bf)(x,v)& = Q^+[f](x,v) \nonumber \\
& = \int_{\mathbb{S}^2} \int_{\mathbb{R}^3} f(x,v')g_0(\bar{v}')[(v-\bar{v})\cdot \nu]_+ \, \mathrm{d}\bar{v} \, \mathrm{d}\nu \nonumber \\
& =  \int_{\mathbb{R}^3} k(v,v_*)f(x,v_*) \, \mathrm{d}v_*.
\end{align}
\end{rmk}
\begin{lem} \label{lem-mya7}
There exists a $C>0$ such that for any $V>0$,
\begin{equation}\label{eq-mya7}
\int_{|v|>V} k(v,v_*) \, \mathrm{d}v \leq C,
\end{equation}
for almost all $|v_*| \leq V$.
\end{lem}
\begin{proof}
A similar calculation to the proof of theorem~2.1 in \cite{arlott07} (see also \cite{molinet77}) allows us to show that in our case, for $v \neq v_*$,
\[k(v,v_*) \leq \frac{2\pi}{|v-v_*|} h \left( \frac{1}{4} \left[ |v-v_*|+ \frac{|v|^2 - |v_*|^2}{|v-v_*|} \right]^2 \right).  \]
Where $h:[0,\infty) \rightarrow \mathbb{R}$ is given by,
\[h(r):= \int_r^\infty  z \sup_{ \{ v \in \mathbb{R}^3: |v|=z\} } g_0(v) \, \mathrm{d}z. \]

Thus we can follow the calculations in \cite[Ex 10.29]{arlotti06} to see that, for $d=3$, for almost all $|v^*| \leq V$ and using assumption \eqref{eq-g0linf5},
\[ \int_{|v|>V} k(v,v_*) \, \mathrm{d}v \leq C\Gamma <\infty, \]
where $C>0$ is some geometric constant and
\[ \Gamma: = \int_0^\infty rh(r^2) \, \mathrm{d}r. \]
\end{proof}
\begin{proof}[Proof of proposition~\ref{prop-linboltz}]
By the definition of $A$ and $B$, \eqref{eq-linboltz} is now,
\begin{equation} \label{eq-linboltzoperatorform}
\begin{cases}
\partial_t f_t(x,v)&=(Af+Bf)(x,v) \\
f_{t=0}(x,v) &= f_0(x,v).
\end{cases}
\end{equation}
By lemma~\ref{lem-tsemi} and equations (10.86), (10.88) of \cite{arlotti06},  \cite[corollary 5.17]{arlotti06} holds for our equation, so we apply  \cite[theorem 5.2]{arlotti06}. Lemma~\ref{lem-mya7} allow us to further apply  \cite[theorem~10.28]{arlotti06}. Therefore we have an honest $C_0$ semigroup of contractions generated by $\overline{A+B}$, which we denote by $G(t)$. Finally by \cite[theorem~3.1.12]{arendt11}, \eqref{eq-linboltzoperatorform} has a unique mild solution for each $f_0 \in L^1(U \times \mathbb{R}^3)$.

It remains to show \eqref{lem-ftmom} and \eqref{lem-ftindb}.
Firstly we find a bound for the operator $B$. Then we prove \eqref{lem-ftmom} for $f_0 \in D(A)$ first before generalising to all $f_0$. Recall \eqref{eq-beta}. As $g_0$ is normalized,
\begin{align} \label{eq-g0betabound}
\int_{\mathbb{S}^2} \int_{\mathbb{R}^3} g_0(\bar{v}) [(v-\bar{v})\cdot \nu]_+ \, \mathrm{d}\bar{v} \, \mathrm{d}\nu  & \leq \int_{\mathbb{S}^2} \int_{\mathbb{R}^3} g_0(\bar{v}) |v-\bar{v}| \, \mathrm{d}\bar{v} \, \mathrm{d}\nu \nonumber \\
& \leq \pi \int_{\mathbb{R}^3} g_0(\bar{v}) (|v| + |\bar{v}|) \, \mathrm{d}\bar{v} \, \mathrm{d}\nu \leq \pi (|v| + \beta).
\end{align}
By (10.6) in \cite{arlotti06} for $f \in L^1(U \times \mathbb{R}^3)$ and by \eqref{eq-g0betabound} recalling \eqref{eq-arlottinu},
\begin{align} \label{eq-bfl1bound}
\int_{U\times \mathbb{R}^3} Bf (x,v) \, \mathrm{d}x \, \mathrm{d}v &= \int_{U\times \mathbb{R}^3} \int_{\mathbb{R}^3} k(v,v_*)f (x,v_*) \, \mathrm{d}v_* \, \mathrm{d}x \, \mathrm{d}v \nonumber \\
& = \int_{U\times \mathbb{R}^3} \nu(x,v)f(x,v) \, \mathrm{d}x \, \mathrm{d}v  \nonumber \\
& = \int_{U\times \mathbb{R}^3} f(x,v)  \int_{\mathbb{S}^2} \int_{\mathbb{R}^3} g_0(\bar{v}) [(v-\bar{v})\cdot \nu]_+ \, \mathrm{d}\bar{v} \, \mathrm{d}\nu  \, \mathrm{d}x \, \mathrm{d}v  \nonumber \\
& \leq \int_{U\times \mathbb{R}^3} f(x,v) \pi (|v| + \beta)\, \mathrm{d}x \, \mathrm{d}v.
\end{align}
Now suppose that $f_0 \in D(A)$. Then  \cite[corollary 5.17]{arlotti06} holds and so we apply  \cite[corollary 5.8]{arlotti06} which gives,
\begin{equation} \label{eq-linboltzduhamelform}
f_t= T(t)f_0 + \int_0^t G(t-\theta)BT(\theta)f_0 \, \mathrm{d}\theta.
\end{equation}
Where $G(t)$ is the contraction semigroup generated by $\overline{A+B}$ and $T(t)$ the contraction semigroup generated by $A$. Hence by \eqref{eq-bfl1bound},
\begin{align*}
\int_{U\times \mathbb{R}^3} f_t(x,v)& (1+|v|) \, \mathrm{d}x \, \mathrm{d}v \\
& \leq \int_{U\times \mathbb{R}^3} T(t)f_0(x,v)(1+|v|)  + \int_0^t G(t-\theta)BT(\theta)f_0 \, \mathrm{d}\theta (x,v)(1+|v|) \, \mathrm{d}x \, \mathrm{d}v \\
& \leq \int_{U\times \mathbb{R}^3} f_0(x,v)(1+|v|) + t f_0 (x,v)(1+|v|) \pi (|v| + \beta) \, \mathrm{d}x \, \mathrm{d}v.
 \end{align*}
Noting that $t\in [0,T]$ and recalling our assumption on $f_0$ in \eqref{eq-f0assmp2}, this is bounded.

Consider a general $f_0$  in \eqref{eq-f0assmp2}, not necessarily in $D(A)$. Suppose for contradiction that \eqref{lem-ftmom} is not true. Hence there exists a $t \in [0,T]$ such that for any $C>0$ there exists an $R>0$ such that,
\[ \int_U \int_{B_R(0)} f_t(x,v)(1+|v|) \, \mathrm{d}v \, \mathrm{d}x \geq C.  \]
$D(A)$ is dense in $L^1(U\times \mathbb{R}^3)$ because it contains, for example, smooth compactly supported functions. Hence for any $n \in \mathbb{N}$ there exists an $f_0^n \in L^1(U\times \mathbb{R}^3)$ such that $f_0^n \in D(A)$, $f_0^n \geq 0$, there exists a $C_1>0$ such that
\begin{equation} \label{eq-f0n2ndmombound}
\int_{U \times \mathbb{R}^3} f_0^n (x,v)(1+|v|^2) \, \mathrm{d}x \, \mathrm{d}v  \leq C_1,
\end{equation}
and that,
\begin{equation} \label{eq-f0f0nbound}
\int_{U \times \mathbb{R}^3} |f_0^n (x,v)-f_0(x,v)|(1+|v|) \, \mathrm{d}x \, \mathrm{d}v  \leq \frac{1}{n}.
\end{equation}
Now define $f_t^n$ to be the solution of \eqref{eq-linboltz} with initial data given by $f_0^n$. In this case \eqref{eq-linboltzduhamelform} gives,
\[ f_t^n= T(t)f_0^n + \int_0^t G(t-\theta)BT(\theta)f_0^n \, \mathrm{d}\theta. \]
The argument above for  $f_0^n \in D(A)$ together with \eqref{eq-f0n2ndmombound} gives that there exists a $C_2>0$ independent of $n$ such that,
\begin{equation*} \label{eq-ftnlower}
 \int_{U \times \mathbb{R}^3} f_t^n (x,v)(1+|v|) \, \mathrm{d}x \, \mathrm{d}v \leq C_2.
\end{equation*}
By our contradiction assumption there exists an $R>0$ such that,
\begin{equation*}
\int_U \int_{B_R(0)} f_t(x,v)(1+|v|) \, \mathrm{d}v \, \mathrm{d}x \geq 2C_2.
\end{equation*}
These two bounds together give that,
\begin{align} \label{eq-ftcontupper}
\int_{U\times \mathbb{R}^3} |f_t(x,v)& -f_t^n(x,v)|(1+|v|) \, \mathrm{d}x \, \mathrm{d}v \nonumber \\
& \geq \int_U \int_{B_R(0)} |f_t(x,v) -f_t^n(x,v)|(1+|v|) \, \mathrm{d}v \, \mathrm{d}x \nonumber \\
& \geq \bigg| \int_U \int_{B_R(0)} f_t(x,v)(1+|v|)\, \mathrm{d}v \, \mathrm{d}x -\int_U \int_{B_R(0)} f_t^n(x,v)(1+|v|) \, \mathrm{d}v \, \mathrm{d}x \bigg| \geq C_2.
\end{align}
However since $G(t)$ is a contraction semigroup it follows for $n > 1/C_2$ by \eqref{eq-f0f0nbound},
\begin{align*}
\int_{U\times \mathbb{R}^3} (f_t(x,v)& -f_t^n(x,v))(1+|v|) \, \mathrm{d}x \, \mathrm{d}v  = \int_{U\times \mathbb{R}^3} G(t)(f_0 -f_0^n)(x,v)(1+|v|) \, \mathrm{d}x \, \mathrm{d}v \\
& \leq \int_{U\times \mathbb{R}^3} (f_0(x,v) -f_0^n(x,v))(1+|v|) \, \mathrm{d}x \, \mathrm{d}v  \leq \frac{1}{n}  < C_2.
\end{align*}
By the same argument,
\[ \int_{U\times \mathbb{R}^3} (f_t^n(x,v) -f_t(x,v))(1+|v|) \, \mathrm{d}x \, \mathrm{d}v < C_2.  \]
Hence we have a contradiction with \eqref{eq-ftcontupper} which completes the proof of \eqref{lem-ftmom}.

To show \eqref{lem-ftindb} fix $t \in [0,T]$. By \eqref{eq-bfl1bound},
\[ \int_{U \times \mathbb{R}^3} Bf_t(x,v)  \, \mathrm{d}x \, \mathrm{d}v \leq \int_{U \times \mathbb{R}^3} f_t(x,v) \pi (|v|+\beta)  \, \mathrm{d}x \, \mathrm{d}v.   \]
By the above calculations $f_t$ has finite first moment so this is finite as required.
\end{proof}

\bibliographystyle{plain}
\bibliography{bib}


\end{document}